\documentclass{article}
\usepackage{amsthm,amsmath,amssymb,amsfonts,graphicx,enumitem}
  \newtheorem{thm}{Theorem}[section]
  \newtheorem{lem}[thm]{Lemma}
  \newtheorem{prop}[thm]{Proposition}
  \newtheorem{cor}[thm]{Corollary}
  \theoremstyle{definition}
  \newtheorem{defn}[thm]{Definition}
  \newtheorem{exm}[thm]{Example}
  \newtheorem{rmk}[thm]{Remark}
  
 \newcommand\ra{\rightarrow}

\newcommand{\lex}{\,\overrightarrow{\times}\,}

 \newcommand\mI{\mathcal{I}}

 \newcommand\s{\subseteq}
 
 \newcommand\supp{\mathrm{Supp}}
  \newcommand\GMV{\mathrm{EMV}}

  \newcommand\lam{\lambda}

\newcommand{\Ker}{\mbox{\rm Ker}}
 \numberwithin{equation}{section}
\def\iff{if and only if }

\def\iff{if and only if }
\begin{document}
\title{\bf{On $EMV$-algebras} }
\author{ Anatolij Dvure\v{c}enskij$^{^{1,2}}$, Omid Zahiri$^{^{3}}$\footnote{Corresponding author } \\
{\small\em $^1$Mathematical Institute,  Slovak Academy of Sciences, \v Stef\'anikova 49, SK-814 73 Bratislava, Slovakia} \\
{\small\em $^2$Depart. Algebra  Geom.,  Palack\'{y} Univer. 17. listopadu 12, CZ-771 46 Olomouc, Czech Republic} \\
{\small\em  $^{3}$University of Applied Science and Technology, Tehran, Iran}\\
{\small\tt  dvurecen@mat.savba.sk\quad   zahiri@protonmail.com} }
\date{}
\maketitle
\begin{abstract}
The paper deals with an algebraic extension of $MV$-algebras based on the definition of generalized Boolean algebras. We introduce a new algebraic structure, not necessarily with a top element,  which is called an $EMV$-algebra and every $EMV$-algebra contains an $MV$-algebra. First, we present basic properties of $EMV$-algebras, give some examples, introduce and investigate congruence relations, ideals and filters on this algebra. We show that each $EMV$-algebra can be embedded into an $MV$-algebra and we characterize $EMV$-algebras either as $MV$-algebras or maximal ideals of $MV$-algebras. We study the lattice of ideals of an $EMV$-algebra and prove that any $EMV$-algebra has at least one maximal ideal. We define an $EMV$-clan of fuzzy sets as a special $EMV$-algebra. We show any semisimple $EMV$-algebra is isomorphic to an $EMV$-clan of fuzzy functions on a set. We consider the variety of $EMV$-algebra and we present an equational base for each proper subvariety of the variety of $EMV$-algebras. We establish a categorical equivalencies of the category of proper $EMV$-algebras, the category of $MV$-algebras with a fixed special maximal ideal, and a special category of Abelian unital $\ell$-groups.
\end{abstract}



{\small {\it Acknowledgement:} AD is thankful for the support by the grants VEGA No. 2/0069/16 SAV and GA\v{C}R 15-15286S  }

\section{ Introduction }
$MV$-algebras were defined by Chang \cite{Cha} as an algebraic counterpart of many-valued reasoning. The principal result of the theory of $MV$-algebras is a representation theorem by Mundici \cite{Mun} saying that there is a categorical equivalence between the category of MV-algebras and the category of unital Abelian $\ell$-groups. Today the theory of $MV$-algebras is very deep and has many interesting connections with other parts of mathematics with many important applications to different areas. For more details on $MV$-algebras, we recommend the monographs \cite{mundici 1, mundici 2}.

$GMV$-algebras, called also pseudo $MV$-algebras \cite{georgescu} or  non-commutative $MV$-algebras \cite{Rach}, are a non-commutative
generalization of $MV$-algebras and the algebraic counterparts of non-commutative many valued logic. Moreover, Galatos and Tsinakis generalized the notion of an $MV$-algebra in the context of residuated lattices to include both commutative and unbounded structures in \cite{Tsinakis} and introduced the notion of generalized $MV$-algebra.
Indeed, a pseudo $MV$-algebra is a bounded integral generalized $MV$-algebra. They extended the relation between unital $\ell$-groups and pseudo $MV$-algebras established by Mundici \cite{Mun} for $MV$-algebras and by
Dvure\v{c}enskij \cite{Dvu2} for pseudo $MV$-algebras. Many other results in these structures can be find in \cite{DiDvTs,Dvu1,Dvu2,Rach2,Shang}. We note that $MV$-algebras are studied in the last period  also in the frames of involutive semirings, see \cite{DiRu}.

There is another way how to generalize the concept of $MV$-algebras considering the definition of generalized Boolean algebras.
In the paper, first we use the definition of  generalized Boolean algebras to extend the concept of $MV$-algebras. We call this structure an $EMV$-algebra. These algebras generalize $MV$-algebras, and in any $EMV$-algebra $M$, the interval $[0,a]$ forms an $MV$-algebra for each idempotent $a\in M$.

The paper is organized as follows. After Preliminaries, Section 2, we introduce $EMV$-algebras in Section 3. We present some examples and find relations between generalized Boolean algebras and $EMV$-algebras. We exhibit basic properties of $EMV$-algebras, and in particular, we show a one-to-one relationship between ideals and congruences of $EMV$-algebras.
In Section 4, we introduce state-morphisms as analogues of finitely additive probability measures, and we show their intimate relationship with maximal ideals. We introduce $EMV$-clans as $EMV$-algebras of fuzzy sets where all algebraic operations are defined by points; they are exactly
semisimple $EMV$-algebras up to isomorphism. We also show that any semisimple $EMV$-algebra can be embedded into an $MV$-algebra.
In Section 5, we show a relationship between ideals and filters of $EMV$-algebras and we show that nevertheless an $EMV$-algebra has not necessarily a top element, it contains at least one maximal ideal. Finally, we prove that every $EMV$-algebra can be embedded into an $MV$-algebra. Moreover, every $EMV$-algebra is either an $MV$-algebra or it can be embedded into an $MV$-algebra as its maximal ideal. This allows us to study subvarieties of the variety of $EMV$-algebras and to present an equational base for each subvariety of $EMV$-algebras. In addition, in Section 6, we present mutual categorical equivalencies of the category of proper $EMV$-algebras with the special category of $MV$-algebras with a fixed maximal ideal having enough idempotents or with a special category of Abelian unital $\ell$-groups.

\section{ Preliminaries}

In the section, we gather some basic notions relevant to $MV$-algebras which will be needed in the next sections. For more details,
we recommend to consult
\cite{Di Nola,mundici 1, mundici 2} for $MV$-algebras.

An $MV$-{\it algebra} is an algebra $(M;\oplus,',0,1)$ (henceforth write simply $M=(M;\oplus,',0,1)$) of type $(2,1,0,0)$, where $(M;\oplus,0)$ is a
commutative monoid with the neutral element $0$ and for all $x,y\in M$, we have:
\vspace{1mm}
\begin{enumerate}[nolistsep]
	\item[(i)] $x''=x$;
	\item[(ii)] $x\oplus 1=1$;
	\item[(iii)] $x\oplus (x\oplus y')'=y\oplus (y\oplus x')'$.
\end{enumerate}

\noindent
In any $MV$-algebra $(M;\oplus,',0,1)$,
we can define also the following  operations:
\[
x\odot y:=(x'\oplus y')',\quad x\ominus y:=(x'\oplus y)'.
\]
\noindent
In addition, let $x \in M$. For any integer $n\ge 0$, we set
$$0.x=0,\quad n.x=(n-1).x\oplus x, \ n\ge 1,$$
and
$$x^0=1, \quad x^n=x^{n-1}\odot x, \ n\ge 1.$$

Moreover, the relation $x\leq y \Leftrightarrow x'\oplus y=1$ is a partial order on $M$ and $(M;\leq)$ is a lattice, where
$x\vee y=(x\ominus y)\oplus y$ and $x\wedge y=x\odot (x'\oplus y)$. Note that, for each $x\in M$, $x'$ is the least element
of the set $\{y\in M\mid x\oplus y=1\}$.
We use $\mathbb{MV}$ to denote the category of $MV$-algebras whose objects are $MV$-algebras and morphisms are $MV$-homomorphisms.
A non-empty subset $I$ of an $MV$-algebra $(M;\oplus,',0,1)$ is called an {\it ideal} of $M$ if $I$ is a down set which is closed under $\oplus$.
The set of all ideals of $M$ is denoted by $I(M)$. It is well known that for each $x,y\in M$, if $I\in I(M)$ and $y,x\ominus y\in I$, then $x\in I$.
For each ideal $I$ of $M$, the relation $\theta_I$ on $M$ defined by $(x,y)\in\theta_I$ if and only if $x\ominus y,y\ominus x\in I$ is a congruence relation on $M$, and $x/I$ and $M/I$ will denote $\{y\in M\mid (x,y)\in\theta_I\}$ and $\{x/I\mid x\in M\}$, respectively. A {\it prime ideal} is an ideal $I\ne M$ of $M$ such that $M/I$ is a linearly ordered $MV$-algebra, or equivalently, for all $x,y \in M$, $x\ominus y \in I$ or $y\ominus x\in I$. The set of all minimal prime ideals of $M$ is denoted by $Min(M)$.
An element $a$ of an $MV$-algebra  $(M;\oplus,',0,1)$ is called a
{\it Boolean element} if there is $b\in M$ such that $a\wedge b=0$ and $a\vee b=1$. The set of all Boolean elements of $M$ forms a Boolean algebra; it is denoted by $B(M)$.

\begin{thm}\label{2.1}{\rm \cite[Thm. 1.5.3]{mundici 1}}
For every element $x$ in an $MV$-algebra $M$, the following conditions are equivalent:
\vspace{1mm}
\begin{itemize}[nolistsep]
\item[{\rm (i)}] $x\in B(M)$;
\item[{\rm (ii)}] $x\vee x'=1$;
\item[{\rm (iii)}] $x\wedge x'=0$;
\item[{\rm (iv)}] $x\oplus x=x$;
\item[{\rm (v)}] $x\odot x=x$;
\item[{\rm (vi)}]  $x\oplus y=x\vee y$ for all $y\in M$;
\item[{\rm (vii)}]  $x\odot y=x\wedge y$ for all $y\in M$.
\end{itemize}
\end{thm}
Let $(M;+,0)$ be a monoid. An element $a\in M$ is called {\it idempotent} if $a+a=a$. The set of all idempotent elements of $M$
is denoted by $\mathcal{I}(M)$.
A monoid $(G;+,0)$ is called {\it partially ordered} if it is equipped with a partial order relation $\leq$ that is compatible with $+$, that is,
$a\leq b$ implies  $x+a+y\leq x+b+y$ for all $x,y\in G$.
A partially ordered monoid $(G;+,0)$ is called a {\it lattice ordered monoid} or simply an $\ell$-{\it monoid}
if $G$ with its partially order relation is a lattice. In a similar way,
a group $(G;+,0)$ is said to be a {\it partially ordered group} if it is a partially ordered monoid. A partially ordered group $(G;+,0)$ is called a {\it lattice ordered group} or simply an $\ell$-{\it group}
if $G$ with its partially order relation is a lattice. An element $x\in G$ is called {\it positive} if $0\leq x$. An element $u$ of an $\ell$-group $(G;+,0)$ is called a {\it strong unit} of $G$ if, for each $g\in G$, there exists $n\in \mathbb{N}$ such that $g\leq nu$. A  couple $(G,u)$, where $G$ is an $\ell$-group and $u$ is a fixed strong unit for $G$, is said to be a {\it unital} $\ell$-{\it group}.

If $(G;+,0)$ is an Abelian $\ell$-group with strong unit $u$, then the interval $[0,u]:=\{g \in G \mid 0\leq g \leq u\}$ with the operations $x\oplus y:=(x+y)\wedge u$ and $x':=u-x$ forms
an $MV$-algebra, which is denoted by $\Gamma(G,u)=([0,u];\oplus,',0,u)$. Moreover, if $(M;\oplus,0,1)$ is an $MV$-algebra, then according to the famous theorem by Mundici, \cite{Mun}, there exists a unique (up to isomorphism)  unital Abelian $\ell$-group
$(G,u)$ with strong $u$ such that $\Gamma(G,u)$ and $(M;\oplus,0,1)$ are isomorphic (as $MV$-algebras).
Let $\mathcal{A}$ be the category of unital Abelian $\ell$-groups whose objects are unital $\ell$-groups and morphisms are unital $\ell$-group morphisms (i.e. homomorphisms of $\ell$-groups preserving fixed strong units). It is important to note that $\mathbb{MV}$ is a variety whereas $\mathcal A$ not because it is not closed under infinite products.  Then $\Gamma: \mathcal{A}\ra \mathbb{MV}$ is a functor between these categories. Moreover, there is another functor from the category of $MV$-algebras to $\mathcal{A}$ sending $M$ to a Chang $\ell$-group induced by good sequences of the $MV$-algebra $M$, which is
denoted by $\Xi:\mathbb{MV}\ra \mathcal{A}$. For more details relevant to these functors, please see \cite[Chaps 2 and 7]{mundici 1}.

\begin{thm}{\rm\cite[Thms 7.1.2, 7.1.7]{mundici 1}}\label{functor}
The composite functors $\Gamma\Xi$ and $\Xi\Gamma$ are naturally equivalent to the identity functors of
$\mathbb{MV}$ and $\mathcal{A}$, respectively. Therefore, the categories $\mathcal{A}$ and $\mathbb{MV}$ are categorically equivalent.
\end{thm}

Recall that a {\em residuated lattice} is an algebra $(L;\vee,\wedge,\cdot,\setminus,/,e)$ of type $(2,2,2,2,2,0)$ such that $(L;\vee,\wedge)$ is a lattice,
$(L;\cdot,e)$ is a monoid, and for all $x,y,z\in L$,
\[x\cdot y\leq z \Leftrightarrow x\leq z/y \Leftrightarrow y\leq x\setminus z.
\]
A residuated lattice is called {\it commutative} if it satisfies the identity $x\cdot y=y\cdot x$ and is called {\em integral} if
it satisfied the identity $x\wedge e=x$.
Galatos and Tsinakis, \cite{Tsinakis}, introduced the concept of a {\em generalized MV-algebra} ($GMV$-algebra) which is a
residuated lattice that satisfies the identities
\[x/((x\vee y)\setminus x)=x\vee y=(x/(x\vee y))\setminus x.
\]
It is well known that bounded commutative integral $GMV$-algebras and $MV$-algebras coincide (see \cite{Jipsen,Tsinakis}).

\section{$EMV$-algebras, Ideals, and Congruences}

In the section, we define $qEMV$-algebras and $EMV$-algebras which form  an important subclass of $qEMV$-algebras. We present some examples and we define subalgebras and homomorphisms. We show that $EMV$-algebras form a variety. Congruences on the class of $EMV$-algebras are in a one-to-one correspondence with the set of ideals. We show that every semisimple $EMV$-algebra can be embedded into an MV-algebra.

\begin{defn}\label{de:GMV}
An algebra $(M;\vee,\wedge,\oplus,0)$ of type $(2,2,2,0)$ is called a \emph{quasi extended $MV$-algebra} ($qEMV$-algebra in short) if it satisfies the
following conditions:
\vspace{1mm}
\begin{itemize}[nolistsep]
\item[]{\rm ($\GMV1$)}\ $(M;\vee,\wedge,0)$ is a distributive lattice with the least element $0$;
\item[]{\rm ($\GMV2$)}\  $(M;\oplus,0)$ is a commutative ordered monoid with neutral element $0$;
\item[]{\rm ($\GMV3$)}\ for all $a,b\in \mI(M)$ such that $a\leq b$, the element
$$
\lambda_{a,b}(x)=\min\{z\in[a,b]\mid x\oplus z=b\}
$$
exists in $M$ for all $x\in [a,b]$, and the algebra $([a,b];\oplus,\lambda_{a,b},a,b)$ is an $MV$-algebra.
\end{itemize}
\vspace{2mm}
We say that an $qEMV$-algebra $(M;\vee,\wedge,\oplus,0)$ has  {\it enough idempotent elements} if, for each $x\in M$, there is $a\in \mI(M)$ such that
$x\leq a$. An {\it extended $MV$-algebra}, an {\it $EMV$-algebra} in short, is a $qEMV$-algebra $(M;\vee,\wedge,\oplus,0)$ which has enough idempotent elements.
\end{defn}
From now on, in this paper, we usually denote $\lambda_{0,b}$ by $\lambda_{b}$.

Now we present some examples of $qEMV$-algebras and $EMV$-algebras, respectively.

\begin{exm}\label{3.1}
\begin{enumerate}
\item[(1)]  Any $MV$-algebra $(M;\oplus,',0,1)$ is an $EMV$-algebra. Let $a,b\in B(M)$.  By Proposition \ref{2.1}, for each
$x,y\in [a,b]$, we have $x\oplus y\geq a\oplus y=a\vee y=y\geq a$ and $x\oplus y\leq x\oplus b=x\vee b=b$, thus
$[a,b]$ is closed under $\oplus$. It can be easily seen that $\overline{x}:=(x'\vee a)\wedge b$ is the least element of the set
$\{z\in [a,b]\mid x\oplus z=b\}$. Moreover, for each $x\in [a,b]$,
\begin{eqnarray*}
\overline{\overline{x}}&=&\Big(((x'\vee a)\wedge b)'\vee a\Big)\wedge b=\Big(((x''\wedge a')\vee b')\vee a\Big)\wedge b \\
&=& \Big( ((x\wedge a')\vee a)\vee b'\Big)\wedge b=(x\vee b')\wedge b=x.
\end{eqnarray*}
Therefore, $([a,b];\oplus,^{-},a,b)$ is an $MV$-algebra (for more details we refer to \cite{dmn}), and so any $MV$-algebra is an $EMV$-algebra.

\item[(2)]  Any generalized Boolean algebra $(M;\vee,\wedge,0)$  (studied also as a Boolean ring, see \cite{LuZa,Kel}) forms an $EMV$-algebra $(M;\vee,\wedge,\oplus,0)$, where $\oplus =\vee$ and if $a\le b$, then $\lambda_{a,b}(x)$ is the unique relative complement of $x$ in the interval $[a,b]$.

\item[(3)] Let $(B;\vee,\wedge)$ be a  generalized Boolean algebra and $(M;\oplus,',0,1)$ be an $MV$-algebra. Then it can be easily
shown that $M\times B$ is an $EMV$-algebra.

\item[(4)] Any bounded $qEMV$-algebra is an $MV$-algebra. Note that if $M$ is a $qEMV$-algebra with the greatest element $1$, then
$M=[0,1]$ and $M$ is an $MV$-algebra.

\item[(5)]  Let $G$ be a non-trivial $\ell$-group. The set of positive elements $G^{+}$ of $G$ with the natural operation $+$ and natural ordering is a $qEMV$-algebra. Since $0$ is the only idempotent element, so $G^+$ is not an $EMV$-algebra.

\item[(6)] Let $\{(M_i;\oplus,',0,1)\}_{i\in I}$ be a family of $MV$-algebras and $S=\{f\in \prod_{i\in I}M_i\mid \supp(f) \mbox{ is finite}  \}$.
Clearly, $S$ is closed under $\vee$, $\wedge$ and $\oplus$. Moreover, if $f\in S$, then $u=(u_i)_{i\in I}$, where $u_i=1$ for all
$i\in \supp(f)$ and $u_i=0$ for all $i\in I\setminus \supp(f)$, is an element of $S$ which is idempotent and $f\leq u$.
It can be easily shown that $S$ is a $qEMV$-algebra and so $S$ is an $EMV$-algebra. We will denote this $qEMV$-algebra by
$\sum_{i\in I} M_i$.

\item[(7)] Let $(M;\oplus,',0,1)$ be an $MV$-algebra and $A$ be any
ideal of $M$. Then similarly to (1), we can see that
$A$ is a  $qEMV$-algebra.

\item[(8)] Let $J$ be an ideal of an $MV$-algebra $(A;\oplus,',0,1)$ and $B$ be a generalized Boolean algebra. Then
$B\times J$ with the pointwise operations forms a $qEMV$-algebra.

\item[(9)] Let $(M;\vee,\wedge,\oplus,0)$ be an $EMV$-algebra. Then it is straightforward to show that
$(\mI(M);\vee,\wedge,0)$ is a generalized Boolean algebra. Moreover,
$M$ is an $MV$-algebra \iff $\mI(M)$ is a Boolean algebra.

\item[(10)] Every finite $EMV$-algebra is an $MV$-algebra.

\item[(11)] Every $EMV$-clan of fuzzy sets on some $\Omega\ne \emptyset$ is an $EMV$-algebra, where all operations are defined by points, see Definition \ref{de:clan} and Proposition \ref{pr:clan} below.
	\end{enumerate}
\end{exm}

\begin{rmk}\label{3.2}
Let $(M;\vee,\wedge,\oplus,0)$ be an $EMV$-algebra.

(i) For each $x,y\in M$, there exist $a,b\in \mI(M)$ such $a\le x\oplus y\leq b$ and so
$x,y\in [a,b]$. Since $([a,b];\oplus,\lambda_{a,b},a,b)$ is an $MV$-algebra, then the element $\lambda_{a,b}(\lambda_{a,b}(x)\oplus y)\oplus y$ is the supremum of $x$ and $y$ taken in the MV-algebra $[a,b]$ and it coincides with   $x\vee y$. Similarly, $x\wedge y = \lambda_{a,b}(\lambda_{a,b}(x)\oplus \lambda_{a,b}(\lambda_{a,b}(x)\oplus y))$. In addition, if $a_0\le x,y \le b_0$ for some $a_0,b_0\in \mI(M)$, then $\lambda_{a,b}(\lambda_{a,b}(x)\oplus y)= x\vee y = \lambda_{a_0,b_0}(\lambda_{a_0,b_0}(x)\oplus y)$.

(ii)  If $a,c, b$ are idempotents with $a\le c \le b$ and $x \in M$ such $x \in [a,b]$, then $x\oplus c = x\vee c$ and $x\wedge c = \lambda_{a,b}(\lambda_{a,b}(x)\oplus \lambda_{a,b}(\lambda_{a,b}(x)\oplus c))$.

(iii) The Riesz Decomposition Theorem holds: If $z\le x\oplus y$, then there are $x_z\le x$ and $y_z\le y$ such that $x=x_z\oplus y_z$. Or if $x_1\oplus x_2=y_1\oplus y_2$, there are four elements $c_{11}, c_{12}, c_{21}, c_{22}\in M$ such that $x_1=c_{11} \oplus c_{12}$, $x_2 =c_{21}\oplus c_{22}$, $y_1 = c_{11} \oplus c_{21}$ and $y_2=c_{21} \oplus c_{22}$. These facts follow from the analogous properties in the $MV$-algebra $[0,a]$, where $a\ge x,y,x_1,x_2,y_1,y_2$.
\end{rmk}

\begin{rmk}\label{GMV-EMV}
	Let $(L;\vee,\wedge,\cdot,\setminus,/,e)$ be a residuated lattice such that $(L;\cdot,e)$ is commutative and $(L;\vee,\wedge)$ is a lattice with the least element $0$.
Then $y/x=x\setminus y$ for all $x,y\in L$, and $x\setminus y$ is usually written $x\ra y$ (see \cite[p. 12]{Tsinakis}). We claim that $L$ is not an $EMV$-algebra. Otherwise, since $0\ra 0$ is the greatest element of $L$, then by Example \ref{3.1}(4), $L$ is an $MV$-algebra. Therefore, from Example \ref{3.1} we get that there exists an $EMV$-algebra which is not a generalized $MV$-algebra in the sense of \cite{Tsinakis}. 	
\end{rmk}

The following proposition shows that the notions of a $qEMV$-algebra and of an $EMV$-algebra can be defined also in a simpler way.

\begin{prop}\label{3.3}
Let $(M;\vee,\wedge,\oplus,0)$ be an algebra of type $(2,2,2,0)$. Then
\begin{itemize}
 \item[{\rm (i)}] $M$ is a $qEMV$-algebra if and only if  $([0,b];\oplus,\lam_b,0,b)$ is an $MV$-algebra for all $b\in \mI(M)$. In such a case, $\lam_{a,b}(x)=\lam_b(x)\vee a$.
 \item[{\rm (ii)}] $M$ is an $EMV$-algebra if and only if
  \subitem{\em ($\GMV_I1$)} $(M;\vee,\wedge,0)$ is a lattice with the least element 0;
  \subitem{\em ($\GMV_I2$)} $(M;\oplus,0)$ is a commutative monoid with neutral element $0$;
  \subitem{\em ($\GMV_I3$)}  for each $x\in M$, there is $b\in \mI(M)$ with $x \le b$ such that $([0,b];\oplus,\lam_b,0,b)$ is an $MV$-algebra.
 \end{itemize}
\end{prop}

\begin{proof}
(i) Let, for all $b\in \mI(M)$, $([0,b];\oplus,\lam_b,0,b)$ be an $MV$-algebra. Take $a,b\in\mI(M)$ such that $a\leq b$.
We show that for each $x\in [a,b]$, $\lam_{a,b}(x)$ exists and $\lam_{a,b}(x)=\lam_b(x)\vee a$.
Indeed, $x\oplus (\lam_b(x)\vee a)=(x\oplus \lam_b(x))\vee (x\oplus a)=b$. Now, let $z\in [a,b]$ such that $x\oplus z=b$. Then
by definition of $\lam_b(x)$, we have $\lam_b(x)\leq z$ and so $\lam_b(x)\vee a\leq z\vee a=z$. Now, we can easily see that $([a,b];\oplus,\lam_{a,b},a,b)$ is an $MV$-algebra. Therefore, $(M;\vee,\wedge,\oplus,0)$ is a $qEMV$-algebra. The proof of the converse is clear.

(ii)  Let ($\GMV_I1$)--($\GMV_I3$) hold.
First we show that $(M;\oplus,0)$ is an ordered monoid. Let $x,y\in M$ be such that $x\leq y$. For each $z\in M$, by the assumption,
there exists $a\in\mI(M)$ such that $y\vee z\leq a$ and $([0,a];\oplus,\lam_a,0,a)$ is an $MV$-algebra and so
$x\oplus z\leq y\oplus z$ (since $x,y,z\in [0,a]$). That is,  $(M;\oplus,0)$ is an ordered monoid. In a similar way, we can show that
$(M;\vee,\wedge)$ is a distributive lattice. Now, by (i), it is enough to show that for all $b\in \mI(M)$, $([0,b];\oplus,\lam_b,0,b)$ is an $MV$-algebra.
Let $b$ be an arbitrary idempotent element of $M$. By the assumption,
there is $u\in \mI(M)$ such that $b\leq u$ and $([0,u];\oplus,\lam_u,0,u)$ is an $MV$-algebra.
It can be easily seen that $\lam_b(x)=b\wedge \lam_u(x)$ for all $x\in [0,b]$, and similarly
to Example \ref{3.1}(1),  $([0,b];\oplus,\lam_b,0,b)$ is an $MV$-algebra. Therefore, $M$ is an $EMV$-algebra.
Clearly, the converse holds.
\end{proof}

Let $(M;\vee,\wedge,\oplus,0)$ be an $EMV$-algebra. Then for all $a\in\mI(M)$, we have a well-known binary operation
$$
x\odot_{_a} y=\lam_a(\lam_a(x)\oplus \lam_a(y)),\quad x,y \in [0,a],
$$
on the $MV$-algebra $([0,a];\oplus,\lam_a,0,a)$.

Inspired by the equivalence in Proposition \ref{3.3}(i), we can define the notion of a $qEMV$-subalgebra also in the following equivalent way.

\begin{defn}\label{3.4}
(i)  Let $(M;\vee,\wedge,\oplus,0)$ be a $qEMV$-algebra. A subset $A\s M$ is called a $qEMV$-{\it subalgebra} of $M$ if $A$ is closed
under $\vee$, $\wedge$, $\oplus$ and $0$ and for each $b\in \mI(M)\cap A$ the set $[0,b]_A:=[0,b]\cap A$ is a subalgebra
of the $MV$-algebra $([0,b];\oplus,\lam_b,0,b)$.  Clearly, the last condition is equivalent to the following condition:
$$\forall\, b\in A\cap \mI(M),\quad \forall\, x\in [0,b]_A,\ \ \min\{z\in [0,b]_A\mid x\oplus z=b\}=\min\{z\in [0,b]\mid x\oplus z=b\},
$$
and due to Proposition \ref{3.3}(i), this also means $\lambda_{a,b}(x)$ is defined in $[a,b]_A$ for all $a,b\in \mathcal I (M)\cap A$ with $a\le b$ and $x\in [a,b]_A$.
If $(M;\vee,\wedge,\oplus,0)$ is an $EMV$-algebra and $A$ is a $qEMV$-subalgebra of $M$ such that
for all $x\in A$,  there is $b\in A\cap \mI(M)$ such that $x\leq b$, then $A$ is called an $EMV$-{\it subalgebra} of $M$.

(ii) Let $(M_1;\vee,\wedge,\oplus,0)$ and $(M_2;\vee,\wedge,\oplus,0)$ be $qEMV$-algebras. A map $f:M_1\ra M_2$ is called a $qEMV$-{\it homomorphism}
if $f$ preserves the operations $\vee$, $\wedge$, $\oplus$ and $0$, and for each $b\in\mI(M_1)$ and for each $x\in [0,b]$, $f(\lam_b(x))= \lam_{f(b)}(f(x))$.

If $M_1$ and $M_2$ are two $EMV$-algebras, then each $qEMV$-homomorphism $f:M_1\ra M_2$ is said to be an $EMV$-{\it homomorphism}.
\end{defn}

\begin{lem}\label{le:hom}
Let $M_1$ and $M_2$ be two $qEMV$-algebras and $f:M_1\ra M_2$ be a $qEMV$-homomorphism.
\vspace{1mm}
\begin{itemize}[nolistsep]
\item[{\rm (i)}] If $B$ is a subalgebra of $M_2$, then $f^{-1}(B)$ is a subalgebra of $M_1$.

\item[{\rm (ii)}] If $M_1$ and $M_2$ are $EMV$-algebras, $f:M_1\to M_2$ is an $EMV$-homomorphism, and $A$ is an $EMV$-subalgebra of $M_1$, then $f(A)$ is an $EMV$-subalgebra of $M_2$.
\end{itemize}
\end{lem}

\begin{proof}
(i) Clearly, $f^{-1}(B)$ is closed under the operations $\oplus$, $\vee$, $\wedge$ and $0$. Let $a\in f^{-1}(B)\cap \mI(M_1)$. We show that $[0,a]\cap f^{-1}(B)$ is a subalgebra of the $MV$-algebra $[0,a]$.
Put $x,y\in [0,a]\cap f^{-1}(B)$. Then clearly, $x\oplus y\in [0,a]\cap f^{-1}(B)$. Also from
$f(\lam_a(x))=\lam_{f(a)}(f(x))$, we get that $\lam_a(x)\in [0,a]\cap f^{-1}(B)$, so $[0,a]\cap f^{-1}(B)$ is a
subalgebra of $[0,a]$. Therefore, $f^{-1}(B)$ is a subalgebra of $M_1$.

(ii) Clearly, $f(A)$ is closed under $\oplus,\vee,\wedge$ and $0$.
Now we have to show that, for each $b\in f(A)\cap\mI(M_2)$, the set $f(A)\cap [0,b]$ is a subalgebra of the $MV$-algebra $[0,b]$.
By definition, since $A$ is an $EMV$-subalgebra of $M_1$, then for each $y\in f(A)$, there is an element $b\in f(A)\cap\mI(M_2)$ such that $y\leq b$. We only need to show that $f(A)\cap [0,b]$ is closed under $\lam_b$. Put $y\in f(A)\cap [0,b]$. Then there exist $a,x\in A$ such that $f(a)=b$ and $f(x)=y$. Let $u\in\mI(M_1)\cap A$ such that $a,x\leq u$. Then $b,y\in [0, f(u)]$ and so
\[\lam_b(y)=\lam_{f(u)}(y)\wedge b=\lam_{f(u)}(f(x))\wedge f(a)=f(\lam_u(x))\wedge f(a)=f(\lam_u(x)\wedge a). \]
Since $A$ is a subalgebra of $M_1$, then $[0,u]\cap A$ is a subalgebra of the $MV$-algebra $[0,u]$ and so $\lam_u(x)\wedge a\in A$, which implies that $\lam_b(y)\in f(A)$. Clearly, $\lam_b(y)\in[0,b]$. Therefore,
$\lam_b(y)\in [0,b]\cap f(A)$. That is, $f(A)$ is an $EMV$-subalgebra of $M_2$.
\end{proof}

\begin{rmk}
	(i) Definition \ref{3.4} yields that each $MV$-homomorphism is an $EMV$-homomorphism, but the converse is not true in general case.
	Indeed, let $(M;\vee, \wedge,\oplus,0)$ be an $EMV$-algebra and $a\in\mI(M)$. Then $([0,a];\oplus,\lam_a,0,a)$ is an $MV$-algebra (and so an $EMV$-algebra). Clearly, the inclusion map $i:[0,a]\ra M$ is an $EMV$-homomorphism. Now, if $(M;\oplus,',0,1)$ is an $MV$-algebra and $a\in M\setminus \{0,1\}$ is its Boolean element, then the inclusion map $i:[0,a]\ra M$ is not an $MV$-homomorphism (since $i(a)\neq 1$).

An $EMV$-homomorphism $f:M\ra N$ is said to be {\em strong} if, for each $b\in \mI(N)$, there exists $a\in\mI(M)$ such that $b\leq f(a)$. Clearly, any $MV$-homomorphism is strong as an $EMV$-homomorphism. Moreover, if $(M;\oplus,',0,1)$ is an $MV$-algebra and $f:M\ra N$ is an $EMV$-homomorphism, then $N$ is an $MV$-algebra and $f$ is an $MV$-homomorphism, since $x\leq f(1)$ for all $x\in N$.
	
	(ii) Let  $f:M\ra N$ be a strong $EMV$-homomorphism and $S$ be a {\it full subset} of $\mI(M)$ (that is, for each $b\in \mI(M)$, there exists $a\in S$ such that $b\leq a$). For each $a\in \mI(M)$, set $f_a:=f\mid_{_{[0,a]}}$. Then we have
	
	(F1)  $\{f_a\mid a\in S\}$ is a family of $MV$-homomorphisms;
	
	(F2) if $a,b\in S$ such that $a\leq b$, then $f_a=f_{b}\mid_{_{[0,a]}}$;
	
	(F3) $\{f(a)\mid a\in S\}$ is a full subset of $\mI(N)$.
	
Conversely, if $(M;\vee, \wedge,\oplus,0)$ and $(N;\vee, \wedge,\oplus,0)$ are two $EMV$-algebras, $S$ is a full subset of $M$ and $\{f_a:[0,a]\ra N\mid a\in S\}$ is a family of maps satisfying conditions (F1)--(F3), then	 the map $f:M\ra N$ defined by $f(x)=f_a(x)$, where $a\in S$ and $x\in [0,a]$, is a strong $EMV$-homomorphism.
\end{rmk}

\begin{prop}\label{3.5}
Let $(M;\vee,\wedge,\oplus,0)$ be a $qEMV$-algebra, $a,b\in \mI(M)$ such that $a\leq b$.
Then  for each $x\in [0,a]$, we have
\vspace{1mm}
\begin{itemize}[nolistsep]
\item[{\rm (i)}]  $\lam_a(x)=\lam_b(x)\wedge a$;
\item[{\rm (ii)}]   $\lam_b(x)=\lam_a(x)\oplus \lam_b(a)$;
\item[{\rm (iii)}] $\lam_{a,b}(x)=\lam_b(x)\vee a$;
\item[{\rm (iv)}] $\lambda_a(x) \le \lambda_b(x)$;
\item[{\rm (v)}] $\lam_b(a)$ is an idempotent, and $\lam_a(a)=0$.
\end{itemize}
\end{prop}

\begin{proof}
Since $([0,b];\oplus,\lam_b,0,b)$ is an $MV$-algebra, $a\in [0,b]$ and $a\oplus a=a$, we get that $a\vee x=a\oplus x$ for all
$x\in [0,b]$.  Let $x\in [0,a]$.

(i)  From $\lam_b(x)\wedge a\in [0,a]$ and  $(\lam_b(x)\wedge a)\oplus x=(\lam_b(x)\oplus x)\wedge (a\oplus x)=b\wedge (a\vee x)=a$,
it follows that $\lam_a(x)\leq \lam_b(x)\wedge a$. Also,
$b=a\oplus \lam_b(a)=(x\oplus \lam_a(x))\oplus \lam_b(a)=x\oplus (\lam_a(x)\oplus \lam_b(a))$, so
$\lam_a(x)\oplus \lam_b(x)\geq \lam_b(x)$. Hence
$(\lam_a(x)\oplus \lam_b(a))\wedge a\geq \lam_b(x)\wedge a$. Since $a$ is a Boolean element of the $MV$-algebra $[0,b]$, then so is
$\lam_b(a)$, which implies that
$\lam_b(x)\wedge a\leq (\lam_a(x)\oplus \lam_b(a))\wedge a=(\lam_a(x)\vee \lam_b(a))\wedge a=\lam_a(x)\wedge a=\lam_a(x)$.
Summing up the above results, we get that
$\lam_a(x)=\lam_b(x)\wedge a$.

(ii) By (i) we have
\begin{align*}
\lam_a(x)\vee \lam_b(a)& =(\lam_b(x)\wedge a)\vee \lam_b(a)\\
&=(\lam_b(x)\vee \lam_b(a))\wedge (a\vee \lam_b(a))\\
&=(\lam_b(x)\vee \lam_b(a))\wedge b=\lam_b(x)\vee \lam_b(a).
\end{align*}
On the other hand, $x,a\in [0,b]$ and $x\leq a$, it follows that $\lam_b(a)\leq \lam_b(x)$ and so $\lam_b(x)\vee \lam_b(a)=\lam_b(x)$.
Therefore, $\lam_b(x)=\lam_a(x)\vee \lam_b(a)$.

(iii) It was proved in Proposition \ref{3.3}(i).

(iv) It follows from (i) or (ii).

(v) Since $[0,b]$ is an $MV$-algebra and $a\in [0,b]$ is an idempotent, $\lambda_b(a)$ is the relative complement of $a$ in $[0,b]$, so it is also an idempotent. The rest statement $\lambda_a(a)=0$ follows from definition of $\lambda_a$.
\end{proof}

\begin{rmk}\label{3.6}
Let $(M;\vee,\wedge,\oplus,0)$ and  $(N;\vee,\wedge,\oplus,0)$ be
$EMV$-algebras and $f:M\ra N$ be a map preserving $\oplus$ and $0$.
If for each $x\in M$, there is $b\in \mI(M)$ such that $x\in [0,b]$ and
$f(\lam_b(x))=\lam_{f(b)}(f(x))$, then $f$ is an $EMV$-homomorphism.
Indeed, if $x,y\in M$, there is
$b\in \mI(M)$ such that $x,y\in [0,b]$. Since $([0,b];\oplus,\lam_b,0,b)$
is an $MV$-algebra,
$x\vee y=\lam_b(\lam_b(x)\oplus y)\oplus y$ and $x\wedge
y=\lam_b(\lam_b(x)\vee \lam_b(y))$. Hence, $f$ preserves $\vee$ and
$\wedge$.
It follows that for each $x\in M$, there is $a\in\mI(M)$ such that $x\leq
a$ and $f:[0,a]\ra [0,f(a)]$ is a homomorphism of
$MV$-algebras.
Now, let $a$ be an arbitrary idempotent element of $M$. Then there exists
$u\in \mI(M)$ such that $a\leq u$ and
$f: [0,u]\ra [0,f(u)]$ is a homomorphism of $MV$-algebras. By Proposition \ref{3.5}(v), for each $x\in
[0,a]$, we have
\[f(\lam_a(x))=f(\lam_u(x)\wedge a)=\lam_{f(u)}(f(x))\wedge
f(a)=\lam_{f(a)}(f(x)).\]
It follows that $f$ is an $EMV$-homomorphism.
\end{rmk}

\begin{thm}\label{3.7}
Let $\mathbb{EMV}$ be the class of $EMV$-algebras. Then $\mathbb{EMV}$ is a variety.
\end{thm}

\begin{proof}
The class $\mathbb{EMV}$ is closed under ${\rm HSP}$.
\end{proof}

\begin{defn}\label{3.8}
Let $(M;\vee,\wedge,\oplus,0)$ be a $qEMV$-algebra. An equivalence relation $\theta$ on $M$ is called a {\it congruence relation} or simply a \emph{congruence} if
it satisfies the following conditions:
\vspace{1mm}
\begin{itemize}[nolistsep]
\item[(i)]  $\theta$ is compatible with $\vee$, $\wedge$ and $\oplus$;

\item[(ii)] for all $b\in \mI(M)$, $\theta\cap ([0,b]\times [0,b])$ is a congruence relation on the $MV$-algebra $([0,b];\oplus,\lam_b,0,b)$.
\end{itemize}
We denote by $\mathrm{Con}(M)$ the set of all congruences on $M$.
\end{defn}

The next proposition makes our work easier when we want to verify that an equivalence relation on a $qEMV$-algebra is a congruence.

\begin{prop}\label{3.9}
An equivalence relation $\theta$ on an $EMV$-algebra $(M;\vee,\wedge,\oplus,0)$ is a congruence if it is compatible with $\vee$, $\wedge$ and $\oplus$, and
for all $(x,y)\in\theta$, there exists $b\in\mI(M)$ such that $x,y\leq b$ and $(\lam_b(x),\lam_b(y))\in\theta$.
\end{prop}

\begin{proof}
Suppose that for each $(x,y)\in\theta$ there exists $u\in\mI(M)$ such that $x,y\leq u$ and $(\lam_u(x),\lam_u(y))\in\theta$. Put $b\in \mI(M)$. We will show that $\theta\cap ([0,b]\times [0,b])$ is a congruence on $([0,b];\oplus,\lam_b,0,b)$. Let $x,y\in [0,b]$ such that $(x,y)\in \theta$. Then  by the assumption, there exists $u\in\mI(M)$ such that $x,y\in [0,u]$ and $(\lam_u(x),\lam_u(y))\in\theta$. Suppose that $v\in\mI(M)$ such that $u,b\leq v$ (for example, $v=u\oplus b$).
By Proposition \ref{3.5}(ii), $\lam_v(x)=\lam_u(x)\oplus \lam_v(u)$ and $\lam_v(y)=\lam_u(y)\oplus \lam_v(u)$. Since $(\lam_u(x),\lam_u(y))\in\theta$ and $\theta$ is compatible with $\oplus$, we get that
$(\lam_v(x),\lam_v(y))\in\theta$. It follows that $(\lam_v(x)\wedge b,\lam_v(y)\wedge b)\in\theta$ and so by Proposition \ref{3.5}(i), $(\lam_b(x),\lam_b(y))\in\theta$.  The proof of the converse is clear.
\end{proof}

Let $\theta$ be a congruence relation on an $EMV$-algebra $(M;\vee, \wedge,\oplus,0)$ and $M/\theta=\{[x]\mid x\in M\}$ (we usually use $x/\theta$ instead of $[x]$). Consider the induced operations $\vee$, $\wedge$ and $\oplus$ on $M/\theta$ defined by
$$
[x]\vee[y]=[x\vee y],\  \  [x]\wedge [y]=[x\wedge y],\   \  [x]\oplus [y]=[x\oplus y],\quad \forall\, x,y\in M.
$$
Clearly, $(\GMV_I1)$ and $(\GMV_I2)$ hold. By Proposition \ref{3.3}(ii), it suffices to show that for each $x/\theta\in M/\theta$, there exists an idempotent element $a/\theta\in M/\theta$ such that $x/\theta\leq a/\theta$ and $([0/\theta,a/\theta];\oplus,\lam_{a/\theta},0/\theta,a/\theta)$ is an $MV$-algebra. Put $x/\theta\in M/\theta$.
There is $a\in \mI(M)$ such that $x\leq a$. Clearly, $x/\theta\leq a/\theta$ and $a/\theta$ is idempotent.
Also, $([0,a];\oplus,\lam_a,0,a)$ is an $MV$ algebra and  $\theta_a:=\theta\cap ([0,a]\times [0,a])$ is its congruence relation,
so $[0,a]/\theta_a$ (with the quotient operations) is an $MV$-algebra, where
\begin{eqnarray}
\label{Eq3.9}
\lam_{a/\theta_a}(x/\theta_a)=\min\{z/\theta_a\mid z\in [0,a],z/\theta_a\oplus x/\theta_a=a/\theta_a\}=\lam_a(x)/\theta_a,\  \  \forall\,
x/\theta_a\in [0,a]/\theta_a.
\end{eqnarray}
First, we show that for all $x/\theta\in [0/\theta,a/\theta]$, $\lam_a(x)/\theta$ is the least element of the set
$\{z/\theta\in [0/\theta,a/\theta]\mid z/\theta\oplus x/\theta=a/\theta\}$. For each $x/\theta\in [0/\theta,a/\theta]$, we have
$x/\theta=x/\theta\wedge a/\theta=(x\wedge a)/\theta$ and $x\wedge a\in [0,a]$. So, we can assume that $x\in [0,a]$.
If $y/\theta\in [0/\theta,a/\theta]$ such that $x/\theta\oplus y/\theta=a/\theta$, then $(x\oplus y,a)\in\theta$ and $x,y,a\in [0,a]$, thus
$(x\oplus y,a)\in \theta_a$, that is $x/\theta_a\oplus y/\theta_a=a/\theta_a$
(which implies that $x/\theta\oplus y/\theta=a/\theta$).
Hence by (\ref{Eq3.9}), $y/\theta_a\geq \lam_a(x)/\theta_a$
and so $y/\theta\geq \lam_a (x)/\theta$. Also, $\lam_a(x)/\theta\oplus x/\theta=a/\theta$. Thus $\lam_{a/\theta}(x/\theta)$ exists and
is equal to $\lam_a(x)/\theta$.  Now, it is straightforward to check that $\lam_{a/\theta}$ satisfies the conditions (1) and (3) in
definition of $MV$-algebras. It follows that  $([0/\theta,a/\theta];\oplus,\lam_{a/\theta},0/\theta,a/\theta)$ is an $MV$-algebra.
Therefore, by Proposition \ref{3.3}(ii), $(M/\theta;\vee,\wedge,\oplus,0/\theta)$ is an $EMV$-algebra, and the mapping $x\mapsto x/\theta$ is an $EMV$-homomorphism from $M$ onto $M/\theta$.

\begin{exm}\label{3.10}
(i) Let $f:M\ra N$ be a $qEMV$-homomorphism. Then $\ker(f)=\{(x,y)\in M\times N\mid f(x)=f(y)\}$ is a congruence on $M$.

(ii) Let $\{(M_i;\oplus,',0,1)\mid i\in I\}$ be a family of $MV$-algebras and $\theta_i$ be a congruence on $M_i$ for all $i\in I$.  Set
$$
\theta=\{(f,g)\mid f,g\in \sum_{i\in I}M_i, (f(i),g(i))\in\theta_{i},\ \forall\, i\in I\}.
$$
Then clearly, $\theta$ is an equivalence relation on $\sum_{i\in I} M_i$ which is compatible with $\vee$, $\wedge$ and $\oplus$.
Let $(f,g)\in\theta$, for some $f,g\in\sum_{i\in I}M_i$. Define $h:I\ra \bigcup_{i\in I} M_i$ by
\[h(i)=\begin{cases}
1, &  i\in \supp(f)\cup \supp(g),\\
0, & \text{otherwise}.
\end{cases}
\]
Clearly, $h$ is an idempotent element of $\sum_{i\in I} M_i$ and $f,g\leq h$.  By Example \ref{3.1}(5), we know that
$\sum_{i\in I} M_i$ is an $EMV$-algebra. Let $b\in\sum_{i\in I} M_i$ be an idempotent.
Consider the $MV$-algebra $([0,b];\oplus,\lam_b,0,b)$.
It can be easily seen that for all $\alpha\in [0,b]$,
\[\lam_b(\alpha)(i) =\begin{cases}
\lam_{b_i}(\alpha(i)), &  i\in \supp(f)\cup \supp(g),\\
0, & \text{otherwise. }
\end{cases}\]
Since $\theta_i$ is a congruence on $M_i$ for all $i\in I$, then
$(\lam_b(f),\lam_b(g))\in\theta$ and so by Proposition \ref{3.9}, $\theta$ is a congruence on the $qEMV$-algebra $\sum_{i\in I} M_i$.
\end{exm}

\begin{defn}\label{3.11}
A non-empty subset $I$ of a $qEMV$-algebra $(M;\vee,\wedge,\oplus,0)$ is called an \emph{ideal} if for each $x,y\in M$
\vspace{1mm}
\begin{itemize}[nolistsep]
\item[(i)] $x\oplus y\in I$ for all $x,y\in I$;
\item[(ii)] $x\leq y$ and $y\in I$ implies that $x\in I$.
\end{itemize}
The set of all ideals of $M$ is denoted by $\mathrm{Ideal}(M)$. Clearly $\{0\},M \in \mathrm{Ideal}(M)$. An ideal $I$ of $M$ is {\it proper} if $I\ne M$.
\end{defn}

Similarly as for $MV$-algebras, see \cite[Prop 1.2.6]{mundici 1}, we have a one-to-one relationship between the set of ideals and the set of congruences on a $qEMV$-algebra.

\begin{thm}\label{3.12}
If $\theta$ is a congruence on an $EMV$-algebra $(M;\vee, \wedge,\oplus,0)$, then $I_\theta:=0/\theta$ is an ideal of $M$.

Conversely, let $I$ be an ideal of an $EMV$-algebra $(M;\vee, \wedge,\oplus,0)$. Then the relation $\theta_I$ defined by
\begin{eqnarray}\label{ER0}
(x,y)\in\theta_I \Longleftrightarrow \Big(\exists\, b\in \mI(M): x,y\leq b\quad \& \quad \lam_b(\lam_b(x)\oplus y),
\lam_b(\lam_b(y)\oplus x)\in I\Big)
\end{eqnarray}
is a congruence on $M$. In addition, the mapping $I\mapsto \theta_I$ is a bijection from the set $\mathrm{Ideal}(M)$ onto the set of congruences on $M$.
\end{thm}

\begin{proof}
Let  $\theta$ be a congruence on an $EMV$-algebra $(M;\vee,\wedge,\oplus,0)$. Then it can be easily shown that
$I_\theta:=0/\theta$ is an ideal of $M$.

Now, let $I$ be an ideal of $M$. Then for each $b\in\mI(M)$, $I\cap [0,b]$ is an ideal of the $MV$-algebra $([0,b];\oplus,\lam_b,0,b)$.
Define a relation $\theta_I$ on $M$ by (\ref{ER0}).
For each $(x,y),(z,w)\in \theta_I$, we can easily see that there exists $u\in \mI(M)$ such that $x,y,z,w\in [0,u]$ and $(x,y),(z,w)\in\theta_{I_u}$,
where $I_u=I\cap [0,u]$. Indeed, put $(x,y),(z,w)\in\theta_I$. Then there are $a,b\in\mI(M)$ such that $x,y\in [0,a]$,
$z,w\in [0,b]$  and $\lam_a(\lam_a(x)\oplus y),\lam_a(\lam_a(y)\oplus x)\in I$ and
$\lam_b(\lam_b(z)\oplus w),\lam_b(\lam_b(w)\oplus z)\in I$. Since $a,b\leq a\oplus b$ and $M$ is an $EMV$-algebra, then
there exists $u\in \mI(M)$ such that $a\oplus b\leq u$. Thus $x,y,z,w\in [0,u]$. By Proposition \ref{3.5},
\begin{eqnarray}\label{ER1}
\lam_u(s)=\lam_a(s)\oplus \lam_u(a),\quad \lam_u(t)=\lam_b(t)\oplus \lam_u(b),\quad \forall\, s\in [0,a],\forall\, t\in [0,b].
\end{eqnarray}
Since $\lam_u(\lam_u(x)\oplus y),\lam_a(\lam_a(x)\oplus y)\in [0,u]$, $\lam_a(\lam_a(x)\oplus y)\in I\cap [0,u]$ and
$I\cap [0,u]$ is an ideal of the $MV$-algebra $([0,u];\oplus,\lam_u,0,u)$, then
from
\begin{eqnarray*}
\lam_u(\lam_u(x)\oplus y)\ominus\lam_a(\lam_a(x)\oplus y)&=&
\lam_u\Big(\lam_u\big(\lam_u(\lam_u(x)\oplus y)\big)\oplus\lam_a(\lam_a(x)\oplus y)\Big)\\
&=& \lam_u\Big(\lam_u(x)\oplus y\oplus\lam_a(\lam_a(x)\oplus y)\Big)\\
&=& \lam_u\Big(\lam_a(x)\oplus \lam_u(a)\oplus y\oplus\lam_a(\lam_a(x)\oplus y)\Big)\\
&=& \lam_u\Big(\big(\lam_a(x) \oplus y\big)\oplus\lam_a(\lam_a(x)\oplus y)\oplus\lam_u(a)\Big)\\
&=& \lam_u(a\oplus\lam_u(a))=\lam_u(u)=0\in I\cap[0,u].
\end{eqnarray*}
It follows that $\lam_u(\lam_u(x)\oplus y)\in I\cap [0,u]$. In a similar way, we can show that
\[\lam_u(\lam_u(y)\oplus x), \lam_u(\lam_u(z)\oplus w), \lam_u(\lam_u(w)\oplus z)\in I\cap [0,u].\]
Hence
$(x,y),(y,z)\in\theta_{I_u}$.  Since $\theta_{I_u}$ is a congruence on the $MV$-algebra $([0,u];\oplus,\lam_u,0,b)$, it proves that $\theta_I$ is a congruence on $M$.

In an analogous way, and using \cite[Prop 1.2.6]{mundici 1}, we have that the mapping $I\mapsto \theta_I$ is a bijection in question.
\end{proof}

\begin{defn}\label{3.13}
Let $I$ be an ideal of an $EMV$-algebra $(M;\vee, \wedge,\oplus,0)$. We denote the $EMV$-algebra $(M/\theta_I;\vee,\wedge, \oplus,0/\theta_I)$ simply by $M/I$,  and $M/I$ is called the {\it quotient $EMV$-algebra} of $M$ induced by $I$.
\end{defn}

In Example \ref{3.1}(2), we showed that any generalized Boolean algebra is an $EMV$-algebra.  Now, let $(M;\vee, \wedge,\oplus,0)$ be an $EMV$-algebra and $I$ be an ideal of $M$ such that for each $x\in M$, there is $a\in\mI(M)$ such that $x\in [0,a]$ and $\lam_a(x)\wedge x\in I$. Then $M/I$ is a generalized Boolean algebra. By definition, it suffices to show that for each $x\in M$, $x/I\oplus x/I=x/I$
(since in this case, the $MV$-algebra $[x/I,y/I]$ is a Boolean algebra).
First, we note that for each $x\in M$, $x\in I$ if and only if $x/I=0/I$.
Let $x\in M$. Then by the assumption, there exists $a\in \mI(M)$ such that
$x\wedge \lam_a(x)\in I$, so in the $MV$-algebra $([0/I,a/I];\oplus,\lam_{a/I},0/I)$ we have
$x/I\wedge \lam_{a/I}(x/I)=x/I\wedge \lam_a(x)/I=0/I$.
Hence, $x/I$ is a Boolean element of this $MV$-algebra and so $x/I\oplus x/I=x/I$.

We recall that a $qEMV$-algebra $M$ is {\it simple} if $M$ possesses only two congruences, and due to Theorem \ref{3.12}, this is equivalent to the condition $\mathrm{Ideal}(M)=\{\{0\},M\}$.

\begin{thm}\label{3.14}
Any simple $EMV$-algebra is a simple $MV$-algebra.
\end{thm}

\begin{proof}
Let $(M;\vee, \wedge,\oplus,0)$ be a simple $EMV$-algebra. We claim that $([0,a];\oplus,\lam_a,0,a)$ is a simple $MV$-algebra for
all $a\in\mI(M)$. Otherwise, there are $a\in \mI(M)\setminus \{0\}$ and an ideal $I$ of the $MV$-algebra $[0,a]$ such that $I\neq [0,a]$ and $\{0\}\neq I$. So, $I$ is an ideal of the $EMV$-algebra $M$ different from $\{0\}$ and $M$, which is a contradiction. Thus,
$([0,a];\oplus,\lam_a,0,a)$ is a simple $MV$-algebra. We show that $a=\max M$. Put $x\in M$. Then there exists $b\in\mI(M)$ such
that  $x,a\leq b$. Since $([0,b];\oplus,\lam_b,0,b)$ is simple and $a\in [0,b]$, then by \cite[Thm. 3.5.1]{mundici 1},
there is $n\in\mathbb{N}$ such that $n.a=b$. From $a\in\mI(M)$ it follows that $a=n.a$, hence $a=b$. That is $x\leq a$.
Therefore, $M=[0,a]$ and so it is a simple $MV$-algebra.
\end{proof}

\begin{exm}
The $qEMV$-algebra in Example \ref{3.1}(5) is a simple $qEMV$-algebra if $G=\mathbb R$.
\end{exm}

\begin{prop}\label{3.15}
	Let $(M;\vee, \wedge,\oplus,0)$ be an $EMV$-algebra and $X\s M$. Then $\langle X\rangle$, the least ideal of $M$ generated by $X$,
	is the set
	\[\{z\in M\mid z\leq x_1\oplus x_2\oplus\cdots \oplus x_n, \  \exists\, n\in\mathbb{N},\exists\, x_1,\ldots, x_n\in M\}.\]
\end{prop}

\begin{proof}
	The proof is straightforward.
\end{proof}

\begin{cor}\label{3.16}
	If $I$ is an ideal of an $EMV$-algebra $(M;\vee, \wedge,\oplus,0)$, then for each $x\in M$,
	\[\langle I\cup \{x\}\rangle=\{z\in M\mid z\leq a\oplus n.x,\ \exists\, a \in I,\ \exists\, n\in\mathbb{N}\}.\]
\end{cor}

\begin{defn}
	An ideal $I$ of an $EMV$-algebra $(M;\vee, \wedge,\oplus,0)$ is \emph{maximal} if, for all $x\in M\setminus I$, $\langle I\cup\{x\}\rangle=M$.
	The set of all maximal ideals of $M$ is denoted by $\mathrm{MaxI}(M)$. In Theorem \ref{3.26}, it will be proved that every proper $EMV$-algebra $M$ possesses at least one maximal ideal.
\end{defn}

\begin{prop}\label{3.17}
If $I$ is a maximal ideal of an $EMV$-algebra $(M;\vee, \wedge,\oplus,0)$, then for each $b\in \mI(M)$, $I_b:=I\cap [0,b]$ is equal to $[0,b]$ or
$I_b$ is a maximal ideal of the $MV$-algebra $([0,b];\oplus,\lam_b,0,b)$.

In addition, if $B= \mathcal I(M)$, then $I\cap \mathcal I(M)$ is a maximal ideal of $B$.
\end{prop}

\begin{proof}
	Let $b\in \mI(M)$ and $I_b\neq [0,b]$.
	Let $x\in [0,b]\setminus I_b$.  Then $\langle I\cup\{x\}\rangle=M$. By Proposition \ref{3.16}, for each $z\in [0,b]$, there exist
	$n\in\mathbb{N}$ and $a\in I$ such that $z\leq a\oplus n.x$ and so $z=z\wedge(a\oplus n.x)$. Since $([0,b];\oplus,\lam_b,0,b)$
	is an $MV$-algebra, then by \cite[Prop. 1.17(1)]{georgescu}, we have
	\[z=z\wedge b\leq (a\oplus n.x)\wedge b\leq (a\wedge b)\oplus (n.x\wedge b)\le(a\wedge b)\oplus n.x.\]
	Hence $z\in \langle I_b\cup\{x\}\rangle$. Therefore, $I_b$ is a maximal ideal of the $MV$-algebra $([0,b];\oplus,\lam_b,0,b)$. 	

Since $I$ is a proper subset of $M$, there are $x \in M \setminus I$ and $a\in B$ such that $x\le a$, whence, $a \notin I\cap \mathcal I(M)$, which says that $I \cap \mathcal I(M)$ is a proper ideal of $B$. Now let $b \in B \setminus I \cap \mathcal I(M)$. Then $b \notin I$. Hence, for each idempotent $a \in B$, there is an element $c\in I$ and an integer $n$ such that $a\le c \oplus n.b= c\vee b$. Then $a= a \wedge (c\vee b)= (a\wedge c) \vee (a\wedge b)$ so that, $\lambda_a(a\wedge b) \le (a\wedge c)\in I$. But according to Proposition \ref{3.5}(v), $\lambda_a(a\wedge b)$ is an idempotent of $M$, and thus $\lambda_a(a\wedge b) \in I\cap \mathcal I(M)$ which yields, $a \le \lambda_a(a \wedge b) \vee b$, and finally, $I\cap B$ is a maximal ideal of $B$.
\end{proof}

The following theorem is in some sense a converse to Proposition \ref{3.17}.

\begin{thm}\label{th:maxI}
Let $M$ be an $EMV$-algebra and let $B=\mathcal I(M)$ be the set of idempotents of $M$. Then $B$ is an $EMV$-subalgebra of $M$.
For every maximal ideal $I$ of $B$, there is a maximal ideal $K$ of $M$ such that $I=K\cap B$. If $I_1,I_2$ are two different maximal ideals of $B$ and if $K_1,K_2$ are maximal ideals of $M$ such that $I_1=K_1\cap B$ and $I_2=K_2 \cap B$, then $K_1 \ne K_2$. Moreover, if $K$ is a maximal ideal of $M$ such that $K\supseteq \langle I \rangle$, where $\langle I\rangle$ is the ideal of $M$ generated by $I$, then $K\cap\mathcal I(M)=I$.
\end{thm}

\begin{proof}
Let $M$ be an $EMV$-algebra and $B$ be the set of all idempotent elements of $M$. Then $B$ is a subalgebra of $M$ (which is a generalized Boolean algebra).

By Proposition \ref{3.16}, for each ideal $I$ of $B$, the ideal of $M$ generated by $I$ is
\[ \langle I\rangle=\downarrow I=\{x\in M\mid x\leq z, \exists\, z\in I\}. \]
Clearly, $\langle I\rangle\cap B=I$. Now, let $I$ be a maximal ideal of $B$. Then $H=\downarrow I$ is an ideal of $M$. Put $x\in M\setminus H$. Since $M$ is an $EMV$-algebra, there exists $a\in B$ such that $x\leq a$. Clearly, $\uparrow a\cap H=\emptyset$ and so the set $S=\{J\in \mathrm{Ideal}(M)\mid H\s J,\ J\cap \uparrow a=\emptyset\}$ is not empty. By Zorn's lemma, $S$ has a maximal element, say $K$. If $Q$ is an ideal of $M$ such that $K\varsubsetneq Q$, then $I\varsubsetneq Q$ and $a\in Q$ which implies $B \subseteq Q$ and $\langle B\rangle \s Q$. It entails $Q=M$. That is, $K$ is a maximal ideal of $M$ and, moreover, it contains $I$. Since $K\cap B$ is an ideal of $B$, from $I\s K\cap B\s B$, it follows that either $K\cap B=I$ or $K\cap B=B$. In the second case, $B\subseteq K$ which implies $K=M$. Therefore, $K\cap B=I$. If $I'$ is another maximal ideal of $B$, $I'\ne I$, the above construction gives a maximal ideal $K'$ of $M$ such that $K'\cap B=I'$. Hence, $K'\ne K$.

Now let $K$ be a maximal ideal of $M$ containing $\langle I \rangle$. Then $K \cap \langle I\rangle\supseteq I$ and the maximality of $I$ in $\mathcal I(M)$ guarantees $K \cap \langle I\rangle= I$.
\end{proof}

Theorem \ref{th:maxI} will be strengthened for $EMV$-algebras satisfying the general comparability theorem in Theorem \ref{th:GCP} below.

Using Theorem \ref{th:maxI}, Theorem \ref{3.14} can be extended as follows. Another application of Theorem \ref{th:maxI} will be done in Theorem \ref{th:state maxI} below.

\begin{thm}\label{th:finite}
If an $EMV$-algebra $M$ has finitely many maximal ideals, then $M$ is an $MV$-algebra.

In particular, if $M$ is linearly ordered, then $M$ has a unique maximal ideal, and $M$ is an $MV$-algebra.
\end{thm}

\begin{proof}
Let $M$ be an $EMV$-algebra and $B$ be the set of all idempotent element of $M$. Then $B$ is a subalgebra of $M$. Our aim is to show that $B$ is a finite set.

Theorem \ref{th:maxI} implies that given a maximal ideal $I$ of $B$, there is a maximal ideal $K$ of $M$ such that $I=K\cap B$. In addition, if $I_1$ and $I_2$ are two different maximal ideals of $B$, and if $K_1,K_2$ are maximal ideals of $M$ such that $K_1\cap B=I_1\ne I_2=K_2\cap I_2$, then $K_1\ne K_2$. This implies that if $M$ has finitely many maximal ideals, then $B$ has also finitely many ideals.

By the proof of \cite[Thm. 2.2]{CoDa}, the generalized Boolean algebra can be embedded into a Boolean algebra of subsets of $\mathrm{MaxI}(B)$. Since $\mathrm{MaxI}(B)$ is finite, then $B$ is also a finite set. Therefore, the element $\bigvee\{a\mid a \in B\}$ is the top element of $B$ as well as of $M$ which implies $M$ is an $MV$-algebra.

Assume that $M$ is linearly ordered. If $I_1$ and $I_2$ are two different ideals, there are $x\in I_1\setminus I_2$ and $y\in I_2\setminus I_1$. If $x\le y$, then $x\in I_2$, an absurd, and if $y\le x$ again we get an absurd. Hence, $I_1=I_2$, $|\mathrm{MaxI}(M)|=1$, and by the first part, $M$ is an $MV$-algebra.
\end{proof}

\begin{thm}\label{3.18}
	Let $I$ be a maximal ideal of an $EMV$-algebra $(M;\vee, \wedge,\oplus,0)$. Then $M/I$ is an $MV$-algebra.
\end{thm}

\begin{proof}
	We claim that $M/I$ is a simple $EMV$-algebra. Let $B$ be an ideal of the $EMV$-algebra $M/I$. Set
	$A:=\{x\in M\mid x/I\in B\}$. Clearly, $0\in A$. If $x,y\in A$, then $x/I,y/I\in B$ and so $(x\oplus y)/I=x/I\oplus y/I\in B$.
	Also, if $x,y\in M$ such that $x\leq y\in A$, then clearly, $x/I\leq y/I\in B$ and so $x/I\in B$, that is $x\in A$. So, $A$ is an ideal of the $EMV$-algebra $(M;\vee, \wedge,\oplus,0)$ which clearly contains $I$. Hence $I=A$ or $A=M$.
	Therefore, $M/I$ is a simple $EMV$-algebra. Now, from Theorem \ref{3.14}, it follows that $M/I$ is an $MV$-algebra.
\end{proof}

Due to \cite[Thm 3.5.1]{mundici 1}, if $A$ is a simple $MV$-algebra, then $A$ is isomorphic to a unique $MV$-subalgebra of the $MV$-algebra of the real interval $[0,1]$, hence in Theorem \ref{3.18}, $M/I$ can be embedded in a unique way into the real interval $[0,1]$ because there is an element $x\in M\setminus I$, so that $x/I>0/I$, and we can assume that the maximal value in $M/I$ is equal to the real number $1$.

\section{State-morphisms, Maximal Ideals, and $EMV$-clans}

In the section we introduce state-morphisms which in the case of $MV$-algebras are exactly extremal states. States are averaging of truth-values in \L ukasiewicz logic and they correspond to an analogue of finitely additive measures in classical logic. We show that state-morphisms are in a one-to-one correspondence with maximal ideals. We present $EMV$-clans as $EMV$-algebras of fuzzy sets where all algebraic operations are defined by points. They are prototypes of semisimple $EMV$-algebras.

According to \cite{Mun2}, a mapping $s$ on an $MV$-algebra $M$ such that $s:M\to [0,1]$ is (i) a {\it state} if (a) $s(1)=1$ and (b) $s(a\oplus b)=s(a)+s(b)$ whenever $a\odot b=0$; (ii) a {\it state-morphism} if $s$ is an $MV$-homomorphism from $M$ into the $MV$-algebra of the real interval $[0,1]$; (iii) an {\it extremal state} if $s = \lambda s_1 +(1-\lambda)s_2$, where $s_1,s_2$ are states on $M$ and $\lambda$ is a real number such that $0<\lambda <1$, then $s_1=s_2=s$. Due to \cite{Mun2} and \cite{Dvu3}, we have that (i) every non-degenerate $MV$-algebra possesses at leat one state; (ii) each state-morphism is a state, and it is an extremal state, and conversely, (iii) every extremal state is a state-morphism.

Inspired by the notion of a state-morphism on $MV$-algebras, we define a state-morphism on an $EMV$-algebra $M$ as follows: A mapping $s:M \to [0,1]$ is a {\it state-morphism} if $s$ is an $EMV$-homomorphism from $M$ into the $EMV$-algebra of the real interval $[0,1]$ such that there is an element $x\in M$ with $s(x)=1$. In the latter case, we can assume that there is an idempotent $a$ such that $s(a)=1$. We define the set $\Ker(s)=\{x\in M \mid s(x)=0\}$, the {\it kernel} of a state-morphism $s$.

The basic properties of state-morphisms are as follows.

\begin{prop}\label{pr:property}
Let $s$ be a state-morphism on an $EMV$-algebra $M$. Then
\vspace{1mm}
\begin{itemize}[nolistsep]
\item[{\rm (i)}] $s(0)=0$;
\item[{\rm (ii)}] $s(a)\in \{0,1\}$ for each idempotent $a \in M$;
\item[{\rm (iii)}] if $x\le y$, then $s(x)\le s(y)$;
\item[{\rm (iv)}] $s(\lambda_a(x))=   s(a)-s(x)$ for each $x\in [0,a]$, $a \in \mathcal I(M)$.
\item[{\rm (v)}] $\Ker(s)$ is a proper ideal of $M$.
\end{itemize}
\end{prop}

\begin{proof}
(i) It is trivial.

(ii) Let $a \in \mathcal I(M)$ and assume $0<s(a)$.  There is the least integer $n_0$ such that $n_0.s(a)=1$ in the $MV$-algebra of the real interval $[0,1]$. Then $s(a) = s(n_0.a) =1$.

(iii) Let $x\le y$. There is an idempotent $a \in M$ such that $y\le a$. If $s(a)=0$, then the restriction $s_a$ of $s$ onto the $MV$-algebra $[0,a]$ is the zero function, so that $s(x)=s(y)$. If $s(a)=1$, then the restriction $s_a$ is a state-morphism on the MV-algebra $[0,a]$, and the monotonicity of $s_a$ in $[0,a]$ implies $s(x)\le s(y)$.

(iv) We have $\lambda_a(x)\oplus x =a$ for each $x\in [0,a]$. If $s(a)=0$, the statement follows from (iii). If $s(a)=1$, the restriction $s_a$ is a state-morphism on the MV-algebra $[0,a]$, and for $s_a$ we have $s_a(\lambda_a(x))=1-s(x)$ which proves (iv).

(v) It follows easily from (i) and (iii).
\end{proof}

\begin{thm}\label{th:state}
{\rm (i)} If $I$ is a maximal ideal of an $EMV$-algebra $M$, then $M/I$ can be embedded in a unique way into the MV-algebra of the real interval $[0,1]$ such that the mapping $s_I: x\mapsto x/I$, $x \in M$, is a state-morphism.

{\rm (ii)} If $s$ is a state-morphism, then $\Ker(s)$ is a maximal ideal of $M$.
In addition, there is a unique maximal ideal $I$ of $M$ such that $s=s_I$.

{\rm (iii)} If for state-morphisms $s_1$ and $s_2$ we have $\Ker(s_1)=\Ker(s_2)$, then $s_1=s_2$.
\end{thm}

\begin{proof}
(i) Let $I$ be a maximal ideal. Due to Theorem \ref{3.18} and Proposition \ref{3.17}, the mapping $s_I$ is in fact an $EMV$-homomorphism. Since $I$ is maximal, there is $x \in M\setminus I$, so that $x/I >0/I$. Without loss of generality, we can assume that the greatest value in $M/I$ is $1$. Hence, $s_I$ is a state-morphism on $M$.

(ii) Conversely, let $s$ be a state-morphism. Put $I=\Ker(s)$. Then $I$ is a proper ideal of $M$. Let $x \notin I$. There is an idempotent $a$ of $M$ such that $x\le a \notin I$ and $s(a)=1$. Then the restriction $s_a$ of $s$ restricted to the $MV$-algebra $[0,a]$ is a state-morphism on $[0,a]$. By \cite[Prop 4.3]{Dvu3}, $I\cap [0,a]= \{y \in [0,a]\mid s_a(y)=0\}$ is a maximal ideal of $[0,a]$. Consequently, there is an integer $n$ such that $\lambda_a(n.x)\in I\cap [0,a]\subseteq I$. Hence, $I$ is a maximal ideal of $I$. The uniqueness of $I$ follows from the \cite{Mun,Dvu3}.

(iii) There is an idempotent $a$ such that $a \notin \{x\in M\mid s_1(x)=0\}$. Then for the restrictions of $s_1$ and $s_2$ onto the $MV$-algebra $[0,a]$, we have by \cite[Prop 4.5]{Dvu3} that $s_1(x)=s_2(x)$ for each $x\in [0,a]$, then $s_1(x)=s_2(x)$ for each $x \in M$.
\end{proof}

Let $\mathcal{SM}(M)$ denote the set of state-morphisms on an $EMV$-algebra $M$. In Theorem \ref{3.26}, it will be proved that every $M$ contains a maximal ideal, so that by Theorem \ref{th:state}(i), $\mathcal{SM}(M)$ is non-void whenever $M\ne \{0\}$.
Using Theorem \ref{th:maxI}, we show that every state-morphism on $\mathcal I(M)$ can be extended to a state-morphism on an $EMV$-algebra $M$.

\begin{thm}\label{th:state maxI}
Every state-morphism $s$ on $\mathcal I(M)$ of an $EMV$-algebra $M$ can be extended to a state-morphism $\hat s$ on $M$.
\end{thm}

\begin{proof}
Let $s$ be a state-morphism on $B=\mathcal I(M)$. By Theorem \ref{th:state}, the set $I=\{x\in \mathcal I(M)\mid s(x)=0\}$ is a maximal ideal on $B$. Theorem \ref{th:maxI} guarantees that there is a maximal ideal $K$ of $M$ such that $K\cap B=I$. Due to Theorem \ref{th:state}(i), there is a unique state-morphism $\hat s$ such that $K=\{x\in M \mid \hat s(x)=0\}$. Since by Proposition \ref{pr:property}(ii), $s(a)\in \{0,1\}$, we see that $\hat s$ is an extension (not necessarily unique) of $s$ onto $M$.
\end{proof}

We note that a sufficient condition for the property ``for every maximal ideal $I$ of the set $\mathcal I(M)$ of idempotent elements of an $MV$-algebra $M$, there is a unique maximal ideal $K$ of $M$ such that $K\cap \mathcal I(M)=I$" is a condition that the $MV$-algebra satisfies the general comparability property, see e.g. \cite[Thm 4.4]{Dvu4}. Therefore, we introduce this notion also for $EMV$-algebras as follows. We say that an $EMV$-algebra $M$ satisfies the {\it general comparability property} if it holds for every $MV$-algebra $([0,a]; \oplus,\lambda_a,0,a)$, i.e. if $a \in \mathcal I(M)$ and $x,y \in [0,a]$, there is an idempotent $e$, $e \in [0,a]$ such that $x\wedge e\le y$ and $y\wedge \lambda_a(e) \le x$. We note that every linearly ordered $EMV$-algebra satisfies the general comparability property, on the other hand, there are $MV$-algebras where the general comparability property fails, see e.g. \cite[Ex 8.6]{Dvu4}.

As we promised above, now we strengthen Theorem \ref{th:maxI} as follows.

\begin{thm}\label{th:GCP}
Let $M$ be an $EMV$-algebra satisfying the general comparability property. If $I_1$ and $I_2$ are maximal ideals of $M$ such that $\mathcal I(M) \cap I_1 = \mathcal I(M)\cap I_2$, then $I_1 = I_2$.
\end{thm}

\begin{proof}
Suppose the converse, i.e. $I_1\ne I_2$. Since $I_1$ and $I_2$ are maximal ideals, there is $x \in I_1\setminus I_2$ and $y\in I_2 \setminus I_1$. Choose an idempotent $a$ such that $x,y \le a$. Then $a \notin I_1\cup I_2$. Then $I_1 \cap [0,a]$ and $I_2\cap [0,a]$ are maximal ideals of $[0,a]$, see Proposition \ref{3.17}. In addition, $I_1 \cap [0,a] \cap \mathcal I(M) = I_2 \cap [0,a] \cap \mathcal I(M)$.  Since the $MV$-algebra $[0,a]$ satisfies the general comparability property, using \cite[Thm 4.4, Cor 4.5]{Dvu4} and Theorem \ref{th:state}, we have $x\in I_1\cap [0,a]= I_2\cap [0,a]\ni y$  which gives $y\in I_1$ and $x\in I_2$, an absurd. Hence, $I_1=I_2$.
\end{proof}

\begin{cor}\label{co:state-m}
Let $M$ be an $EMV$-algebra satisfying the general comparability property. Then every state-morphism on $\mathcal I(M)$ can be extended to a unique state-morphism on $M$.
\end{cor}

\begin{proof}
Let $s$ be a state-morphism on $\mathcal I(M)$. By Theorem \ref{th:state maxI}, there is an extremal state $s_1$ on $M$ which is an extension of $s$. If there is another state-morphism $s_2$ on $M$ which extends $s$, then $I_i =\{x \in M\mid s_i(x)=0\}$ is by  Theorem \ref{th:state}(iii) a maximal ideal of $M$ for $i=1,2$. Since $I_1 \cap \mathcal I(M)=I_2\cap \mathcal I(M)$, Theorem \ref{th:GCP} implies, $I_1=I_2$, which finally gives, Theorem \ref{th:state}, $s_1 = s_2$.
\end{proof}

Now, there is a natural question ``under which suitable condition on an ideal $I$ of an $EMV$-algebra $(M;\vee, \wedge,\oplus,0)$, the
quotient $EMV$-algebra induced by $I$, $M/I$, is an $MV$-algebra''?

\begin{lem}\label{3.19}
	Let $(M;\vee, \wedge,\oplus,0)$ be an $EMV$-algebra, $I$ an ideal of $M$,  $b,d\in \mI(M)$ such that $b\leq d$, and let $I_b$ and $I_d$ be ideals of the $MV$-algebras $[0,b]$ and $[0,d]$, respectively. If $x,y\in [0,b]$ such that
	$x/I_d=y/I_d$, then $x/I_b=y/I_b$.
\end{lem}

\begin{proof}
	Let $x,y\in [0,b]$ such that $x/I_d=y/I_d$. Then $\lam_d(\lam_d(x)\oplus y),\lam_d(\lam_d(y)\oplus x)\in I_d\s I$.
	We will show that $\lam_b(\lam_b(x)\oplus y),\lam_b(\lam_b(y)\oplus x)\in I_b$.
	\begin{align*}
	\lam_b(\lam_b(x)\oplus y)&=\lam_b\big((\lam_d(x)\wedge b)\oplus y\big) \\
	& = \lam_b\big((\lam_d(x)\oplus y)\wedge (b\oplus y)\big) \\
	& = \lam_b\big((\lam_d(x)\oplus y)\wedge b\big) ,\mbox{   since $y\leq b$ } \\
	& = \lam_d\big((\lam_d(x)\oplus y)\wedge b\big) \wedge b \\
	& = \big(\lam_d(\lam_d(x)\oplus y)\vee \lam_d(b)\big) \wedge b \\
	& = \big(\lam_d(\lam_d(x)\oplus y)\wedge b\big)\vee \big(\lam_d(b)\wedge b\big) \\
	& = \lam_d(\lam_d(x)\oplus y)\wedge b\in I_b.
	\end{align*}
	In a similar way, we can see that $\lam_b(\lam_b(y)\oplus x)\in I_b$. Therefore, $x/I_b=y/I_b$.
\end{proof}

The following equivalencies on the induced order for a quotient $EMV$-algebra are used in Theorem \ref{3.20}.

Let $I$ be an ideal of an $EMV$-algebra $(M;\vee, \wedge,\oplus,0)$ and $x,y\in M$. Then
\begin{align*}
x/I\leq y/I & \Leftrightarrow x/I=x/I\wedge y/I=(x\wedge y)/I \\
 & \Leftrightarrow \lam_b(\lam_b(x)\oplus (x\wedge
y)),\lam_b(\lam_b(x\wedge y)\oplus x)\in I,\quad \exists\, b\in\mI(M) \\
 & \Leftrightarrow \lam_b(\lam_b(x)\oplus (x\wedge y))\in I ,\quad
\exists\, b\in\mI(M) \\
 & \Leftrightarrow \lam_b\big((\lam_b(x)\oplus x)\wedge (\lam_b(x)\oplus
y)\big)\in I ,\quad \exists\, b\in\mI(M) \\
 & \Leftrightarrow \lam_b\big(b\wedge (\lam_b(x)\oplus y)\big)\in I ,\quad
\exists\, b\in\mI(M) \\
 & \Leftrightarrow \lam_b(\lam_b(x)\oplus y)\in I,\quad \exists\,
b\in\mI(M). \\
\end{align*}

\begin{thm}\label{3.20}
	Let $I$ be an ideal of $EMV$-algebra $(M;\vee, \wedge,\oplus,0)$. Then $M/I$ is an $MV$-algebra if and only if
	there exists $a\in\mI(M)$ such that $\lam_b(a)\in I$ for all $b\in\mI(M)$ greater than $a$.
\end{thm}

\begin{proof}
	Let $M/I$ be an $MV$-algebra. Then there exists $a\in M$ such that $x/I\leq a/I$ for all $x\in M$. Since $M$ is an $EMV$-algebra,
there is $b\in\mI(M)$ such that $a\leq b$ and so $a/I=b/I$. Thus, without loss of generality we can assume that $a\in \mI(M)$.
	Let $b$ be an arbitrary element of $\mI(M)$ greater than $a$. Since $a/I$ is the maximum of $M/I$, then $a/I=b/I$ and so
	there is $d\in \mI(M)$ such that $a,b\leq d$ and $\lam_d(\lam_d(a)\oplus b),\lam_d(\lam_d(b)\oplus a)\in I$.
	Since $a,b\in [0,b]$, by Lemma \ref{3.19}, we get that $\lam_b(a)=\lam_b(\lam_b(b)\oplus a)\in I$.

	Conversely, let $x\in M$. Then there exists $b\in \mI(M)$ such that $x,a\in [0,b]$. Since
	$\lam_b(\lam_b(x)\oplus a)\leq \lam_b(a)\in I$, then $x/I\leq a/I$. Therefore, $M/I$ is an $MV$-algebra. 	
\end{proof}

From Theorem \ref{3.18} and Theorem \ref{3.20} we get that each maximal ideal satisfies the condition in Theorem \ref{3.20}.

We say that an $EMV$-algebra $M$ is {\it semisimple} if $\text{Rad}(M):=\bigcap\{I \mid I \in \mathrm{MaxI}(M)\}=\{0\}$; the set $\text{Rad}(M)$ is said to be the {\it radical} of $M$.

In what follows, we show that every generalized Boolean algebra is semisimple.

\begin{lem}\label{le:GBA}
Let $M\ne \{0\}$ be a generalized Boolean algebra. Then:
\vspace{1mm}
\begin{itemize}[nolistsep]
\item[{\rm (i)}] An ideal $I$ of $M$ is maximal if and only if, for each $a \notin I$ and each $b\in M$ with $a<b$, $\lambda_b(a)\in I$.
\item[{\rm (ii)}] $M$ is a semisimple $EMV$-algebra.
\end{itemize}
\end{lem}

\begin{proof}
If $M$ has the top element, $M$ is a Boolean algebra and the statement is well-known from the theory of Boolean algebras. Thus let us assume that $M$ has no top element. In generalized Boolean algebras we have $x\oplus y = x\vee y$.

(i) Let $I$ be a maximal ideal of $M$, and let $a,b \in M$ such that $a\notin I$ and $a<b$. Maximality of $I$ entails that $\lambda_b(a)\in \langle I\cup\{a\}\rangle$, by Corollary \ref{3.16}, there is an integer $n$ and $x \in I$ such that $\lambda_b(a)\le x \oplus n.a=x\vee a\le x$. Then $\lambda_b(a)=\lambda_b(a)\wedge (x\vee a)=(\lambda_b(a)\wedge x)\vee (\lambda_b(a)\wedge a) = \lambda_b(a)\wedge x\le x$ which yields $\lambda_b(a) \in I$.

Conversely, let an ideal $I$ of $M$ satisfy conditions of (i). Choose $x \in \langle I\cup\{a\}\rangle \setminus I$ and let $b\in M$ be greater than $x$. Then $\lambda_b(x)\in I \subseteq \langle I\cup\{a\}\rangle$. Since $b=x\oplus \lambda_b(x)$, we have $b \in \langle I\cup\{a\}\rangle$. This is true for each $b\in M$ with $b>x$, therefore, we have $\langle I\cup\{a\}\rangle=M$, so that $I$ is maximal.

(ii) First we have to note that every generalized Boolean algebra $M\ne \{0\}$ possesses at least one maximal ideal, as it will be proved in Theorem \ref{3.26} below.

Let $x \in \mathrm{Rad}(M)$. If $x>0$, using Zorn's lemma, we have that there is a maximal filter $F$ of $M$ containing $x$. By Theorem \ref{3.26} below, the set $I_F=\{\lambda_a(z) \mid z\in F,\, a\in \mathcal I(M),\, z \le a\}$ is a maximal ideal of $M$. Let $b$ be an idempotent such that $b \notin I_F$ and $x\le b$. Then $\lambda_b(x)\in I_F$ and $x\in I_F$. Hence, $b= x\oplus \lambda_b(x)\in I_F$ which is absurd. Consequently, $x =0$.
\end{proof}

An important family of $EMV$-algebras is a family of $EMV$-clans of fuzzy sets which as we show below are only semisimple $EMV$-algebras.

\begin{defn}\label{de:clan}
A system $\mathcal T\subseteq [0,1]^\Omega$ of fuzzy sets of a set $\Omega\ne \emptyset$ is said to be an $EMV$-{\it clan} if
\vspace{1mm}
\begin{enumerate}[nolistsep]
\item[(i)] $0_\Omega \in \mathcal T$ where $0_\Omega(\omega)=0$ for each $\omega \in \Omega$;
\item[(ii)] if $a \in \mathcal T$ is a characteristic function, then (a) $a-f \in \mathcal T$ for each $f\in \mathcal T$ with $f(\omega)\le a(\omega)$ for each $\omega \in \Omega$, (b) if $f,g \in \mathcal T$ with $f(\omega),g(\omega)\le a(\omega)$ for each $\omega \in \Omega$, then $f\oplus g \in \mathcal T$, where $(f\oplus g)(\omega) = \min\{f(\omega)+g(\omega),a(\omega)\}$, $\omega \in \Omega$, and $a$ is a characteristic function from $\mathcal T$;
\item[(iii)] for each $f \in \mathcal T$, there is a characteristic function $a \in \mathcal T$ such that $f(\omega)\le a(\omega)$ for each $\omega \in \Omega$;
\item[(iv)] given $\omega \in \Omega$, there is $f \in \mathcal T$ such that $f(\omega)=1$.
\end{enumerate}
\end{defn}

\begin{prop}\label{pr:clan}
Any $EMV$-clan $\mathcal T$ can be organized into an $EMV$-algebra of fuzzy sets where all operations are defined by points.

{\rm (i)} Let $f,g \in \mathcal T$ and $f,g \le a,b$, where $a,b$ are characteristic functions from $\mathcal T$. Then  $(f\odot g)(\omega) = \max\{f(\omega)+g(\omega)-a(\omega),0\}= \max\{f(\omega)+g(\omega)-b(\omega),0\}$ and $f\odot g  \in \mathcal T$. Similarly, $f\ast g= f\odot (a-g)=f\odot (b-g)\in \mathcal T$.

{\rm (ii)} If $f,g \in \mathcal T$ and $f(\omega)\le g(\omega)$ for each $\omega \in \Omega$, then $g-f\in \mathcal T$.

{\rm (iii)} If $f,g\in \mathcal T$, then $f\wedge g, f\vee g \in \mathcal T$, where $(f\wedge g)(\omega)=\min\{f(\omega),g(\omega)\}$ and $(f\vee g)(\omega)=\max\{f(\omega),g(\omega)\}$ for each $\omega \in \Omega$

In addition, given $\omega \in \Omega $, the mapping $s_\omega(f):=f(\omega)$, $f \in \mathcal T$, is a state-morphism and if $I_\omega=\{f \in \mathcal T\mid f(\omega)=0\}$, then $I_\omega$ is a maximal ideal of $\mathcal T$.
\end{prop}

\begin{proof}
We assert that $\mathcal T$ is an $EMV$-algebra of fuzzy sets. By (iv) we have that if $f,g \in \mathcal T$, there is a characteristic function $a\in \mathcal T$ such that $f,g\le a$. If $b\in \mathcal T$ is another characteristic function such that $a\le b$, we have

\[(f\oplus_a g)(\omega)=\begin{cases}
f(\omega)+g(\omega) & \text{ if } f(\omega)+g(\omega)\le a(\omega)\\
a(\omega) & \text{ if }  f(\omega)+g(\omega) > a(\omega),
\end{cases}\quad \omega \in \Omega,
\]
and
\[(f\oplus_b g)(\omega)=\begin{cases}
f(\omega)+g(\omega) & \text{ if } f(\omega)+g(\omega)\le b(\omega)\\
b(\omega) & \text{ if }  f(\omega)+g(\omega) > b(\omega),
\end{cases}\quad \omega \in \Omega.
\]

If $a(\omega)=0$, then $f(\omega)=g(\omega)=0$ and $(f \oplus_a g)(\omega)= 0= (f \oplus_b g)(\omega) $. If $a(\omega)=1$, then $b(\omega)=1$ and $(f \oplus_a g)(\omega)=  (f \oplus_b g)(\omega)$.  Hence, if $f,g \le u,v$, where $u,v$ are characteristic functions from $\mathcal T$, there is a characteristic function $c\in \mathcal T$ such that $u,v\le c$. Then $f\oplus_u g=f\oplus_c g = f\oplus_v g$, and the binary operation $\oplus$ does not depend on chosen characteristic functions $a,b,u,v,c\in \mathcal T$ dominating $f,g$, and $\oplus $ is a total binary operation such that $(\mathcal T;\oplus, 0_\Omega)$ is a commutative ordered monoid. It is easy to see that,  for $f \in \mathcal T$, we have $f\oplus f=f$ iff $f$ is a characteristic function. Finally $\lambda_a(f)=a-f$ whenever $f\le a$ and $a\in \mathcal T$ is a characteristic function. So that $([0,a];\oplus,\lambda_a,0,a)$ is an $MV$-algebra of fuzzy sets. Consequently, $(\mathcal T; \vee,\wedge,\oplus, 0_\Omega)$ is an $EMV$-algebra.

(i) We can define $f\odot _a g$ in the similar but dual way as we defined already $f\oplus_a g$ if $f,g \in [0,a]$, and if $f,g \in [0,b]$, then $f\odot_a g = f\odot_b g$.

(ii) Let $f,g \in \mathcal T$ and $f\le g$. There is a characteristic function $a\in \mathcal T$ such that $f,g$ belong to the $MV$-algebra $[0,a]$. Then $g-f= g \odot (a-f)\in [0,a]$ and similarly $f\odot_a (a-g)=f\odot_b (b-g)$, so we can define $f\ast g:= f\odot_a (a-g)$.

(iii) Let $f,g \in \mathcal T$, and let $a\in \mathcal T$ be a characteristic function such that $f(\omega),g(\omega)\le a(\omega)$, $\omega \in \Omega$. Then $f\vee g =(f\ast g)\oplus g\in \mathcal T$, $f\wedge g = f\odot ((a-f)\oplus g)\in \mathcal T$, and $(f\wedge g)(\omega)=\min\{f(\omega),g(\omega)\}$ and $(f\vee g)(\omega)=\max\{f(\omega),g(\omega)\}$ for each $\omega \in \Omega$.

The mapping $s_\omega$ is a homomorphism from $\mathcal T$ into the $MV$-algebra $[0,1]$. Since, for $\omega\in \Omega$, there is an $f\in \mathcal T$ such that $f(\omega)=1$,  $s_\omega$ is a state-morphism.

In view of $I_\omega= \{f\in \mathcal T \mid s_\omega(f)=0\}$, due to Theorem \ref{th:state}(ii), $I_\omega$ is a maximal ideal of $\mathcal T$ for each $\omega \in \Omega$.
\end{proof}

\begin{thm}\label{th:semis}
An $EMV$-algebra $(M;\vee, \wedge,\oplus,0)$ is semisimple if and only if $M$ is isomorphic to an $EMV$-clan of fuzzy functions on some $\Omega \ne \emptyset$.
\end{thm}

\begin{proof}
Put $\Omega =\mathcal{SM}(M)$, and for each $x \in M$, let $\hat x: \mathcal{SM}(M)\to [0,1]$ be a mapping such that $\hat x(s):=s(x)$, $s \in \mathcal{SM}(M)$. Put $\widehat{M}:=\{\hat x\mid x \in M\}$. We assert that $\widehat{M}$ is an $EMV$-clan. Clearly $\hat 0$ is the zero function. If $a\in \mathcal I(M)$, then by Proposition \ref{pr:property}(ii), $\hat a$ is a characteristic function.

We have $\hat x \le \hat y$ iff $x\le y$.  Indeed, there is an idempotent $a \in M$ such that $x,y \le a$. By Proposition \ref{3.17}, for each maximal ideal $I$ on $M$, the set $I\cap [0,a]$ is either $[0,a]$ or a maximal ideal of $[0,a]$. Then $\bigcap\{I\cap [0,a]\mid I \in \mathrm{MaxI}(M)\}=\{0\}$, so that each $[0,a]$ is a semisimple $MV$-algebra. Then $\hat x\le \hat y$ iff $s(x)\le s(y)$ for each state-morphism $s$ on the semisimple $MV$-algebra $[0,a]$, consequently, $x\le y$.

Then the mapping $x\mapsto \hat x$ preserves $0,\vee,\wedge, \oplus$, if $\hat x \le \hat a$, then $x\le a$, so that $\hat \lambda_a(x)=\hat a - \hat x \in \widehat M$, and $[\hat 0,\hat a]$ is an $MV$-algebra. In other words, $\widehat M$ is an $EMV$-clan of fuzzy sets, and the mapping $x\mapsto \hat x$ is an $EMV$-isomorphism from $M$ onto $\widehat M$.

Given a state-morphism $s$ on $M$, there is an element $x\in M$ such that $s(x)=1$. Then $\hat x(s)=1$.

Conversely, let $\phi:M\to  \mathcal T$ be an $EMV$-isomorphism from $M$ onto an $EMV$-clan $\mathcal T$. If we set $I_\omega:=\{f\in \mathcal T \mid f(\omega)=0\}$, then by Proposition \ref{pr:clan}, $I_\omega$ is a maximal ideal of $\mathcal T$. Then $\bigcap\{I\mid I\in \mathrm{MaxI}(M)\} \subseteq \bigcap \{I_\omega\mid \omega \in \Omega\}=\{0_\Omega\}$, so that $\mathcal T$ is semisimple, and consequently so is $M$.
\end{proof}

A special type of an $EMV$-clan is a clan of fuzzy sets: We say that a system $\mathcal C$ of fuzzy sets of a set $\Omega \ne \emptyset$ is a {\it clan} if (i) $1_\Omega \in \mathcal C$, where $1_\Omega(\omega)=1$ for each $\omega \in \Omega$, (ii) if $f \in \mathcal C$, then $f':=1-f\in \mathcal C$, and (iii) if $f,g \in \mathcal C$, then $f\oplus g\in \mathcal C$, where $(f\oplus g)(\omega)=\min\{f(\omega)+g(\omega),1\}$, $\omega \in \Omega$.  Then $(\mathcal C; \oplus, ', 0_\Omega,1_\Omega)$ is an $MV$-algebra where all $MV$-operations are defined by points. In addition, $\min\{f,g\},\max\{f,g\}\in \mathcal C$ whenever $f,g \in \mathcal C$.  It is clear, that any clan can be understood as a bounded $EMV$-clan. In addition, an $EMV$-clan $\mathcal T$ is a clan iff $1_\Omega \in \mathcal T$.

We note that if $\mathcal T$ is any system of fuzzy sets of $\Omega\ne \emptyset$, then there is a minimal clan $\mathcal C_0(\mathcal T)$ containing $\mathcal T$. In particular, if $\mathcal T$ is an $EMV$-clan, then $\mathcal C_0(\mathcal T)$ is the least clan of fuzzy sets on $\Omega$ containing $\mathcal T$.

\begin{cor}\label{3.21}
	Any semisimple $EMV$-algebra can be embedded into an $MV$-algebra.
\end{cor}

\begin{proof}
	Let $(M;\vee, \wedge,\oplus,0)$ be a semisimple  $EMV$-algebra.
We present two types of the proofs.

(i) By the proof of Theorem \ref{th:semis},  the natural mapping
	\[ \phi:M\ra \prod_{I\in\mathrm{MaxI}(M)} M/I,\quad \phi(x)=(x/I)_{_{I\in \mathrm{MaxI}(M)}},\   x\in M, \]
is an embedding.
	By Theorem \ref{3.18}, $M/I$ is an $MV$-algebra for all $I\in \mathrm{MaxI}(M)$, so
	$\prod_{I\in\mathrm{MaxI}(M)} M/I$ is an $MV$-algebra.

(ii) By Theorem \ref{th:semis}, there is an $EMV$-clan $\mathcal T$ of fuzzy sets on $\Omega \ne \emptyset$ such that $M$ and $\mathcal T$ are isomorphic. Then $\mathcal C_0(\mathcal T)$ is the least  clan containing $\mathcal T$. Then $M$ can be embedded into the $MV$-algebra $\mathcal C_0(\mathcal T)$.
\end{proof}

\begin{rmk}\label{re:clan}
As we have already said, if $I$ is a maximal ideal of $M$, $M/I$ can be understood as an $MV$-subalgebra of the $MV$-algebra of the real interval $[0,1]$. If $\{I_\alpha\}_\alpha$ is a set of maximal ideals of $M$ such that $\bigcap_\alpha I_\alpha =\{0\}$, the embedding mapping $\phi: M \to \prod_\alpha M/I_\alpha$ defined by $ \phi(x)=(x/I_\alpha)_\alpha$ gets an $EMV$-clan $\phi(M)$ of fuzzy sets on $\Omega =\{I_\alpha\}_\alpha$, and the direct product $\prod_\alpha M/I_\alpha$ defines a clan of fuzzy sets on $\Omega$ such that $\mathcal C_0(\phi(M))\subseteq \prod_\alpha M/I_\alpha$. Question: Is $\mathcal C_0(\phi(M))$ equal to $\prod_\alpha M/I_\alpha$?
\end{rmk}

In the next example we show that the answer to the question posed in Remark \ref{re:clan} could be negative.

\begin{exm}\label{ex:clan}
Let $\Omega$ be an infinite set and $\mathcal T$ be the set of characteristic functions of all finite subsets of $\Omega$. Then $\mathcal T$ is a generalized Boolean algebra that is not a Boolean algebra, more precisely $\mathcal T$ is an $EMV$-clan that is not a clan. It contains a system of maximal ideals $\{I_\omega \mid \omega \in \Omega\}$ such that $\bigcap_\omega I_\omega =\{0\}$ and $\mathcal C_0(\mathcal T)\varsubsetneq  \prod_\omega \mathcal T/I_\omega$.

Let $\mathcal T'$ be the system of fuzzy sets on $\Omega$ such that $f\in \mathcal T'$ if and only if there is a finite subset $A$ such that $f\le \chi_A$. Then $\mathcal T'$ is an $EMV$-clan of fuzzy sets, and $\mathcal C_0(\mathcal T')$ consists of fuzzy sets $f$ on $\Omega$ such that either  there is a finite subset $A$ of $\Omega$ such that $f\le \chi_A$ or there is a co-finite subset $A$ such that $f\ge \chi_A$. In addition, $\mathcal T'$ is a maximal ideal of $\mathcal C_0(\mathcal T')$.
\end{exm}

\begin{proof}
It is evident that $\mathcal C_0(\mathcal T)$ consists of  all characteristic functions of all finite or co-finite subsets of $\Omega$. Given $\omega \in \Omega$, let $I_\omega =\{\chi_A\mid  A \subseteq \Omega,\,\omega \notin A\}$. Lemma \ref{le:GBA}(i) implies that $I_\omega$ is a maximal ideal of $\mathcal T$ and $\bigcap\{I_\omega \mid \omega \in \Omega\} =\{\chi_\emptyset\}$. The mapping $\chi_A \mapsto \{\chi_A/I_\omega\mid \omega \in \Omega\} \in \prod_{\omega \in \Omega}\mathcal T/I_\omega = 2^\Omega$ is an embedding of $\mathcal T$ into the clan $2^\Omega$. Clearly, $\mathcal C_0(\mathcal T)\varsubsetneq  2^\Omega$.

It is easy to verify that $\mathcal T'$ is an $EMV$-clan. Let $f \in \mathcal C_0(\mathcal T')\setminus \mathcal T'$. Then there is a co-finite subset $A$ of $\Omega$ such that $f\le \chi_A$, which yields, $1-f\ge 1_\Omega -\chi_A=\chi_{\Omega \setminus A}$, i.e. $1_\Omega-f \in \mathcal T'$, proving $\mathcal T'$ is a maximal ideal of $\mathcal C_0(\mathcal T')$.
\end{proof}

Now we generalize the latter example and \cite[Thm. 2.2]{CoDa}.

\begin{thm}\label{th:clan1}
Let $\mathcal T$ be an $EMV$-clan of fuzzy sets of $\Omega \ne \emptyset$ and let $1_\Omega \notin \mathcal T$. Then the minimal clan $\mathcal C_0(\mathcal T)$ generated by $\mathcal T$ is the set
$$\mathcal T_0=\{f \in [0,1]^\Omega \mid \exists\, f_0\in \mathcal T\ \text{\rm such that either } f=f_0 \ \text{\rm  or } f = 1-f_0\}.
$$
In addition, $\mathcal T$ is a maximal ideal of $\mathcal C_0(\mathcal T)$.
\end{thm}

\begin{proof}
Clearly, $\mathcal T\subseteq \mathcal T_0$, $1_\Omega \in \mathcal T_0$, and if $f \in \mathcal T_0$, then $1-f\in \mathcal T_0$. Now we show that $\mathcal T_0$ is closed under $\oplus$. Let $f,g \in \mathcal T_0$. We have the following three cases: (i) $f=f_0\in \mathcal T$, $g=g_0\in \mathcal T$. Then trivially $f\oplus g \in \mathcal \mathcal T_0$. (ii) Let $f=1-f_0$ and $g=1-g_0$ for some $f_0,g_0 \in \mathcal T$. Then $f\oplus g = (1-f_0)\oplus (1-g_0)= 1-(f_0\odot g_0)$ and $f_0\odot g_0 \in \mathcal T$ by Proposition \ref{pr:clan}(i). (iii) $f=f_0\in \mathcal T$ and $g=1-g_0$, where $g_0 \in \mathcal T$. Then
\[f\oplus g= 1-((1-f_0)\odot g_0)= 1-(g_0\odot (1-f_0))=1-(g_0\odot (1-(g_0\wedge f_0)))=1-(g_0-(g_0\wedge f_0)).\]
Since $f_0\wedge g_0\in \mathcal T$, $g_0-(f_0\wedge g_0)\in \mathcal T$ by Proposition \ref{pr:clan}(ii), so that $f\oplus g \in \mathcal T_0$. Then $\mathcal C_0(\mathcal T)=\mathcal T_0$.

Now we show $\mathcal T$ is a maximal ideal of $\mathcal C_0(\mathcal T)$. Let $g\le f\in \mathcal T$ for $g \in \mathcal C_0(\mathcal T)$. If $g \notin \mathcal T$, then $f\ge g=1-g_0$ for some $g_0\in \mathcal T$ which implies $f_0\odot g_0=1_\Omega \in \mathcal T$, absurd. Hence, $g\in \mathcal T$, and $\mathcal T$ is an ideal of $\mathcal C_0(\mathcal T)$.

Let $g\in \mathcal C_0(\mathcal T) \setminus \mathcal T$. Then $g=1-g_0$ for some $g_0 \in \mathcal T$. Whence, $g_0=1-(1-g_0) \in \mathcal T$ and $\mathcal T$ is a maximal ideal of $\mathcal C_0(\mathcal T)$.
\end{proof}

\begin{cor}\label{co:clan1}
Every proper semisimple $EMV$-algebra can be embedded into an $MV$-algebra as its maximal ideal.
\end{cor}

\begin{proof}
It follows from Theorem \ref{th:clan1} and Theorem \ref{th:semis}.
\end{proof}

\section{Filters, Ideals and Representation of $EMV$-algebras}

One of the main purposes of this part is to show that any $EMV$-algebra has at least one maximal ideal. For this reason, first we define the notion of a filter of an $EMV$-algebra, showing that for each filter $F$ there is an ideal related to it. Since any bounded $EMV$-algebra with top element $1$ is an $MV$-algebra, the existence of a maximal ideal is an easy application of Zorn's lemma if $0\ne 1$. Therefore, we will prove the existence of a maximal ideal in any {\it proper $EMV$-algebra} $M$, that is, $M$ has no  maximal element. Therefore, in a proper $EMV$-algebra $M$, for each $x\in M$, we can find an idempotent element $a$ such that $x<a$. In particular, we show that every $EMV$-algebra can be embedded into an MV-algebra, and we show a basic result saying that every $EMV$-algebra is either an $MV$-algebra or it can be embedded into an $MV$-algebra as its maximal ideal.

\begin{lem}\label{le:x<y}
Let $(M;\vee, \wedge,\oplus,0)$ be an $EMV$-algebra. For all $x,y\in M$, we define
$$
x\odot y=\lam_a(\lam_a(x)\oplus \lam_a(y)),
$$
where $a\in\mI(M)$ and $x,y\leq a$. Then $\odot:M\times M\ra M$ is an order preserving, associative well-defined binary operation on $M$ which does not depend on $a\in \mI(M)$ with $x,y \le a$.

In addition, if $x,y \in M$, $x\le y$, then
\begin{equation}\label{eq:x<y}
y \odot \lambda_a(x)=y\odot \lambda_b(x)
\end{equation}
for all idempotents $a,b$ of $M$ with $x,y\le a,b$, and
\begin{equation}\label{eq:x<y1}
y= (y \odot \lambda_a(x))\oplus x.
\end{equation}
If $x,y \in [0,a]$ for some idempotent $a\in M$, then
\begin{equation}\label{eq:x<y2}
x\odot \lambda_a(y)=x\odot \lambda_a(x\wedge y) \quad \mbox{ and }\quad  x=(x\wedge y) \oplus (x\odot \lambda_a(y)).
\end{equation}
Moreover, a binary operation $\ast$ on $M$ defined by $x\ast y = x\odot \lambda_a(y)$ is correctly defined for all $x,y \in M$.

An element $x\in M$ is idempotent if and only if $x\odot x = x$.

\end{lem}

\begin{proof}
It suffices to show that $\odot$ is
well defined. Put $x,y\in M$. We show that for all $a,b\in \mI(M)$ such that $x,y\leq a,b$, we have
$\lam_a(\lam_a(x)\oplus \lam_a(y))=\lam_b(\lam_b(x)\oplus \lam_b(y))$. That is $x\odot_{_a}y=x\odot_{_b}y$.
Indeed, take $c\in \mI(M)$ such that $a,b\leq c$. Then by Proposition \ref{3.5}, we have
\begin{eqnarray*}
    \lam_c(\lam_c(x)\oplus \lam_c(y))&=&\lam_c\big(\lam_a(x)\oplus \lam_c(a)\oplus \lam_a(y)\oplus \lam_c(a)\big) \\
	&=& \lam_c\big(\lam_a(x)\oplus \lam_a(y)\oplus \lam_c(a)\oplus \lam_c(a)\big)=\lam_c\big(\lam_a(x)\oplus \lam_a(y)\oplus \lam_c(a)\big)\\
	&=& \lam_c\big(\lam_a(x)\oplus \lam_a(y)\big)\odot_{_c} \lam_c(\lam_c(a))=\lam_c\big(\lam_a(x)\oplus \lam_a(y)\big)\odot_{_c} a\\
	&=& \lam_c\big(\lam_a(x)\oplus \lam_a(y)\big)\wedge a=\Big(\lam_a\big(\lam_a(x)\oplus \lam_a(y)\big)\oplus \lam_c(a)\Big)\wedge a \\
	&=& \Big(\lam_a\big(\lam_a(x)\oplus \lam_a(y)\big)\vee \lam_c(a)\Big)\wedge a=\lam_a\big(\lam_a(x)\oplus \lam_a(y)\big)\wedge a \\
	&=& \lam_a\big(\lam_a(x)\oplus \lam_a(y)\big).
\end{eqnarray*}
In a similar way, we can show that $\lam_c(\lam_c(x)\oplus \lam_c(y))=\lam_b(\lam_b(x)\oplus \lam_b(y))$.

To prove associativity, let $x$, $y$ and $z$ be elements of an $EMV$-algebra $(M;\vee, \wedge,\oplus,0)$. Put $c\in\mI(M)$ such that $x,y,z\leq c$. Then by definition of $\odot$, we have $x\odot y=\lam_c(\lam_c(x)\oplus \lam_c(y))$, $y\odot z=\lam_c(\lam_c(y)\oplus \lam_c(z))$ and
both belong to $[0,c]$. It follows that $(x\odot y)\odot z=\lam_c(\lam_c(x\odot y)\oplus\lam_c(z))$ and
$x\odot (y\odot z)=\lam_c(\lam_c(x)\oplus \lam_c(y\odot z))$, which implies that $(x\odot y)\odot z=x\odot (y\odot z)$. Therefore, in any $EMV$-algebra $(M;\vee, \wedge,\oplus,0)$ the binary operation $\odot$ is associative.

In a similar way, we can see that $\odot$ is order preserving.

Now let $x\le y$, $x,y \le a,b$ for some $a,b \in \mathcal I(M)$. There is an idempotent $c$ such that $a,b\le c$. Check and use Proposition \ref{3.5}(ii)
$$
y \odot \lambda_c(x)= y \odot (\lambda_a(x)\oplus \lambda_c(a))=y \odot (\lambda_a(x)\vee \lambda_c(a))= (y \odot \lambda_a(x)) \vee (y\odot\lambda_c(a))
$$
and
$$y\odot\lambda_c(a) \le y \odot \lambda_c(y)=0$$
because  for $y\le a\le c$ we have $\lambda_c(a)\le \lambda_c(y)$.  This implies $y\odot  \lambda_a(x)=y\odot \lambda_c(x)$. In the same way we have $y\odot  \lambda_b(x)=y\odot \lambda_c(x)$ establishing
$y\odot  \lambda_a(x)=y\odot \lambda_b(x)$.

To prove (\ref{eq:x<y1}), it is enough to calculate it in the $MV$-algebra $[0,a]$.

Now let $x,y \le a$ for some $a \in \mathcal I(M)$. Then $x \odot \lambda_a(x\wedge y)= x \odot (\lambda_a(x)\vee \lambda_a(y))= (x\odot \lambda_a(x))\vee (x\odot \lambda_a(y))= x \odot \lambda_a(y)$.

Using (\ref{eq:x<y1}), we can establish (\ref{eq:x<y2}).

The property $x$ is an idempotent of $M$ iff $x\odot x$ follows from definition of the operation $\odot$.
\end{proof}

For any integer $n\ge 1$ and any $x$ of an $EMV$-algebra $M$, we can define
$$ x^1 =1, \quad x^n=x^{n-1}\odot x, \ n\ge 2,$$
and if $M$ has a top element $1$, we define also $x^0=1$.

\begin{defn}\label{3.22}
	A non-empty subset $F$ of an $EMV$-algebra $(M;\vee, \wedge,\oplus,0)$ is called a \emph{filter} if it satisfies the following conditions:
\vspace{1mm}
	\begin{itemize}[nolistsep]
		\item[(i)] for each $x,y\in M$, if $x\leq y$ and $x\in F$, then $y\in F$ (formally $F$ is an upset);
		\item[(ii)] for each $x,y\in F$, $x\odot y\in F$.
	\end{itemize}
The set of all filters of $M$ is denoted by $\mathrm{Fil}(M)$. Clearly, $M \in \mathrm{Fil}(M)$, and a filter $F$ is {\it proper} if $F\ne M$. A proper filter which cannot be a proper subset of another proper filter of $M$ is said to be {\it maximal}, and we denote by $\mathrm{MaxF}(M)$ the set of maximal filters of $M$. By Zorn's lemma,  $\mathrm{MaxF}(M)\ne \emptyset$.
\end{defn}

Let $(M;\vee, \wedge,\oplus,0)$ be a proper $EMV$-algebra. Then there is a non-zero idempotent element $a\in M$. We can easily see that $\uparrow a$ is a filter of the $EMV$-algebra $M$, which is clearly a proper subset of $M$. In a similar way, we can see that $\uparrow a\setminus\{a\}$ is also a
proper filter of $M$.

\begin{prop}\label{3.23}
	Let $F$ be a filter of a proper  $EMV$-algebra $(M;\vee, \wedge,\oplus,0)$. Then the set
	\[ I_F:=\{\lam_a(x)\mid x\in F, a\in \mI(M), x\leq a\} \]
	is an ideal of $M$.
\end{prop}

\begin{proof}
	First, we note that for each $x\in M$, we have
	\begin{eqnarray*}
	x\in I_F &\Leftrightarrow &  \exists\, w \in F\ \exists\, a\in \mI(M):\ w\leq a,\ \lam_a(w)=x \\
	&\Leftrightarrow & \exists\, a\in \mI(M): x\leq a,\ \lam_a(x)\in F.
	\end{eqnarray*}
Let $x,y\in M$ such that $x\in I_F$ and $y\leq x$. Then there exists $a\in \mI(M)$ such that $x\leq a$ and $\lam_a(x)\in F$.
Since $x,y\in [0,a]$, then $\lam_a(x)\leq \lam_a(y)$ and so by the assumption, $\lam_a(y)\in F$. It follows that $y\in I_F$.
Now, suppose that $x,y\in I_F$. Then there exist $a,b\in \mI(M)$ such that $x\leq a$ and $y\leq b$ and $\lam_a(x)\in F$ and $\lam_b(y)\in F$.
Put $c\in\mI(M)$ such that $a,b\leq c$. Then by Proposition \ref{3.5},
$\lam_c(x),\lam_c(y)\in F$ and so $\lam_c(x)\odot \lam_c(y)\in F$.
Since $\lam_c(x),\lam_c(y)\leq c$, $\lam_c(x)\odot \lam_c(y)=\lam_c(x\oplus y)$, hence $x\oplus y\in I_F$.
Therefore, $I_F$ is an ideal of $M$.
\end{proof}

\begin{prop}\label{3.24}
Let $F$ be a proper filter of an  $EMV$-algebra $(M;\vee, \wedge,\oplus,0)$.

{\rm (i)} For each $x\in M$, the least filter $\lfloor F\cup\{x\}\rfloor$ of $M$ containing
		$F\cup\{x\}$ is the set $\{z\in M\mid z\geq y\odot x^n,\ \exists\, n\in\mathbb{N}, \exists\, y\in F\}$.

{\rm (ii)}  $F$ is a maximal filter if and only if, for each $x \notin F$, there are an integer $n$ and an idempotent $b$ with $x\le b$ such that $\lambda_b(x^n) \in F$.
\end{prop}

\begin{proof}
	The proof of the first part is straightforward.

For the second one, let $F$ be a maximal filter and $x \notin F$. By (i), there are an integer $n$ and an element $c \in F$ such that $0=c\odot x^n$. There is an idempotent $b\ge x,c$, so that $c,x$ are in the MV-algebra $[0,b]$. Then $c\odot x^n$ can be calculated in $[0,b]$, so that $c\le \lambda_b(x^n)$ and $\lambda_b(x^n)\in F$.

The converse follows easily from (i).
\end{proof}

\begin{lem}\label{3.25}
		Let $F$ be a proper filter of a proper  $EMV$-algebra $(M;\vee, \wedge,\oplus,0)$.
\vspace{1mm}
		\begin{itemize}[nolistsep]
		\item[{\rm (i)}] For each $a\in F\cap\mI(M)$, $a\notin I_F$.
		\item[{\rm (ii)}] If $a\in\mI(M)\cap F$, then for all $b\in \mI(M)$ such that $a\leq b$, $\lam_b(a)\in I_F$.
		\item[{\rm (iii)}]  If $F$ is a maximal filter of $M$, then for each $a\in\mI(M)$,  $a\notin I_F$ implies $a\in F$.
		\item[{\rm (iv)}] If $J$ is a maximal ideal of $M$, then
		\begin{equation}\label{eq3.25}
		\forall\, a\in \mI(M),\ a\notin J\Longrightarrow (\forall\, b\in\mI(M),\ a<b)\  \lam_b(a)\in J.
		\end{equation}
		\item[{\rm (v)}] If $J$ is an ideal of $M$ satisfying {\rm(\ref{eq3.25})}, then
$$
F_J:=\{\lam_a(x)\mid x\in J,\ a\in \mI(M)\setminus J,\ x< a \}
$$
is a filter of $M$.
\end{itemize}
\end{lem}

\begin{proof}
	(i) Otherwise, $a\in I_F$ implies that there exists $b\in \mI(M)$ such that $a\leq b$ and $\lam_b(a)\in F$ and so
$\lam_b(a),a\leq b$  and $\lam_b(a),a\in F$. Thus $0=a\odot\lam_b(a)\in F$, which is
a contradiction.

(ii) It follows from definition of $I_F$.

(iii) Let $a\in\mI(M)$ such that $a\notin I_F$. Then by definition, for all $b\in\mI(M)$ with $a\leq b$, $\lam_b(a)\notin F$.
If $a\notin F$, then $\lfloor F\cup\{a\}\rfloor=M$ (since $F$
is maximal) and so there exist $n\in\mathbb{N}$ and $x\in F$ such that
$0\geq x\odot a^n=x\odot_u a^n$ for some $u\in\mI(M)$ such that $x,a\leq u$.
Also, $([0,u];\oplus,\lam_u,0,u)$ is an $MV$-algebra and $a$ is a Boolean element of it, so $a^n=a$ and
$\lam_u(a)$ is the greatest element of $[0,u]$ satisfying the equation $z\odot a=0$. It follows that
$x\leq \lam_u(a)$ and $\lam_u(a)\in F$, which is a contradiction.

(iv) Let $a$ be an idempotent element of $M$ such that $a\notin J$. For each $b\in \mI(M)$ with $a\leq b$, we have
$\lam_b(a)\in \langle J\cup\{a\}\rangle$, so by Corollary \ref{3.16},  there exist $n\in\mathbb{N}$ and $x\in J$ such that
$\lam_b(a)\leq x\oplus n.a$. Hence we get
\begin{eqnarray*}
\lam_b(a)&=&\lam_b(a)\wedge (x\oplus n.a)\leq (\lam_b(a)\wedge x)\oplus (\lam_b(a)\wedge n.a) \\
&=& \lam_b(a)\wedge x,\mbox{ \quad since $\lam_b(a)$ and $a$ are Boolean elements of the $MV$-algebra $[0,b]$ }
\end{eqnarray*}
whence $\lam_b(a)\leq x\in J$. Therefore, $\lam_b(a)\in J$ (since $J$ is an ideal).

(v) Let $x,y\in M$. If $y\geq x\in F_J$, then there exists $a\in\mI(M)\setminus J$ such that $x<a$ and $\lam_a(x)\in J$. Let $b\in \mI(M)$ such that $a,y<b$. Then $\lam_b(y)\leq\lam_b(x)=\lam_a(x)\oplus \lam_b(a)$. By the assumption,  $\lam_b(a)\in J$, so $\lam_a(x)\oplus \lam_b(a)\in J$, which implies that $\lam_b(y)\in J$. Thus $y\in F_J$ and $F_J$ is an upset.
Moreover, if $x,y\in F_J$, then there exist $a,b\in\mI(M)$ such that $x<a$ and $y<b$ and $\lam_a(x),\lam_b(y)\in J$.
Let $c\in\mI(M)$ such that $a,b<c$. Then by the assumption, $\lam_c(a),\lam_c(b)\in J$ and hence by Proposition \ref{3.5}, we have
$\lam_c(x)=\lam_a(x)\oplus \lam_c(a)\in J$ and $\lam_c(y)=\lam_b(y)\oplus \lam_c(b)\in J$. It follows that
$\lam_c(x)\oplus \lam_c(y)\in J$. Now, from definition of $F_J$, we have $\lam_c\big(\lam_c(x)\oplus \lam_c(y)\big)\in F_J$.
That is, $x\odot y=x\odot_{_c}y\in F_J$. Therefore, $F_J$ is a filter of $M$.
\end{proof}

\begin{thm}\label{3.26}
	Any proper $EMV$-algebra has at least one maximal ideal. In addition, if $F$ is a maximal filter of $M$, then
$$
I_F:=\{\lam_a(x)\mid x\in F,\ a\in \mI(M),\ x\leq a\}
$$
is a maximal ideal of $M$.
\end{thm}

\begin{proof}
	Let $(M;\vee, \wedge,\oplus,0)$ be a proper $EMV$-algebra. Then $M$ has a non-zero idempotent $a$ and $\uparrow a$ is a proper filter of $M$.
By Zorn's lemma, we can easily see that, $S$, the set of filters of $M$ not containing $0$, has at least one maximal element which is clearly a maximal filter of $M$, $F$ say. Set
$$
I_F:=\{\lam_a(x)\mid x\in F,\ a\in \mI(M),\ x\leq a\}.
$$
By Proposition \ref{3.23}, $I_F$ is an ideal of $M$. Since $F\neq\emptyset$, there exists $a\in\mI(M)\cap F$. By Lemma \ref{3.25}(i), $a\notin I_F$ and so $I_F\neq M$. We claim that $I_F$ is a maximal ideal. Let $J$ be an ideal of $M$ containing $I_F$. For each $a\in\mI(M)$, if $a\notin J$, then $a\notin I_F$, so by Lemma \ref{3.25}(iii), $a\in F$. It follows that $\lam_b(a)\in I_F\s J$ for all $b\in\mI(M)$ greater than $a$. By Lemma \ref{3.25}(v), $F_J:=\{\lam_a(x)\mid x\in J,\ a\in \mI(M) \setminus J,\ x<a \}$ is a filter of $M$.
	Let $x$ be an arbitrary element of $F$. Since $J$ is a proper ideal, then there is an idempotent element $v\in M$	which is not in $J$ (otherwise, $J=M$). Put $w\in \mI(M)$ such that $x,v<w$. Then $w\notin J$ and by definition,	$\lam_w(x)\in I_F\s J$, hence $x\in F_J$. That is $F\s F_J$. Since $F$ is a maximal filter, then $F_J=F$ or $F_J=M$. From $F_J=M$, we get that $0\in F_J$ and so there are $x\in J$ and $a\in\mI(M)$ such that $x<a$ and $0=\lam_a(x)$ which is a contradiction (since $[0,a]$ is an $MV$-algebra). So, $F_J=F$. Let $x\in J$. Then there is $a\in\mI(M)$ such that $x<a$ and $\lam_a(x)\in F_J=F$. It follows that $x=\lam_a(\lam_a(x))\in I_F$, which simply shows that $I_F=J$. Therefore, $I_F$ is a maximal ideal
of $M$.
\end{proof}

\begin{defn}
	A proper ideal $I$ of an $EMV$-algebra $(M;\vee, \wedge,\oplus,0)$ is called {\em prime} if, for each $x,y\in M$,
	$x\wedge y\in M$ implies that $x\in M$ or $y\in M$. We denote by $\mathcal P(M)$ the set of prime ideals of $M$.
\end{defn}

We note that (i) in the next statement was already proved in Proposition \ref{3.17}, here we proved it in a different way using e.g. the Riesz Decomposition Property.

\begin{prop}\label{pr:2}
Let  $I$ be a maximal ideal of an $EMV$-algebra $M$.

{\rm (i)} For each idempotent $a\in M$, either $I\cap [0,a]=[0,a]$ or $I\cap [0,a]$ is a maximal ideal of the $MV$-algebra $[0,a]$.

{\rm (ii)} $I$ is a prime ideal.
\end{prop}

\begin{proof}
(i) If $a \in I$, then $I\cap [0,a]=[0,a]$. If $a\notin I$, then $I\cap [0,a]$ is a proper ideal of the $MV$-algebra $[0,a]$. Let $x \in [0,a]\setminus (I\cap [0,a])$. Since $I$ is maximal, there are an integer $n$ and an element $c \in I$ such that $a \le c \oplus n.x$. Applying the Riesz Decomposition Property, Remark \ref{3.2}(iii), there is $c_1\le c$ and $x_1,\ldots,x_n\le x$ such that $a= c_1 \oplus x_1\oplus \cdots \oplus x_n\le c_1 \oplus n.x\le a$ because $c_1 \in I\cap [0,a]$. Then $\lambda_a(n.x) \in I\cap [0,a]$ which proves by Proposition \ref{3.24} means that $I\cap [0,a]$ is a maximal ideal of the $MV$-algebra $[0,a]$.

(ii) To prove $I$ is prime, let $x\wedge y \in I$ for some $x,y \in M$. If $M$ has the greatest idempotent, $M$ is an $MV$-algebra and the statement is well known. Thus assume $M$ is proper. Since $I$ is a proper ideal of $M$, there is an element $x_0 \notin I$ and there is an idempotent $a$ of $M$ such that $x_0,x,y\le a$ and of course, $a \notin I$. Then by (i), $I\cap [0,a]$ is a maximal ideal of the $MV$-algebra $[0,a]$ and $x,y \in [0,a]$.  Then $x\wedge y \in I\cap [0,a]$ which proves $x\in I\cap [0,a]$ or $y \in I\cap [0,a]$. Consequently, $I$ is prime.
\end{proof}

\begin{prop}\label{3.28}
Every prime ideal of a proper $EMV$-algebra $(M;\vee, \wedge,\oplus,0)$ is contained in a unique maximal ideal of $M$.
\end{prop}

\begin{proof}
	Let $I$ be a prime ideal of an $EMV$-algebra $M$. Put $a\in M \setminus I$.
	For each $b\in\mI(M)$ with $a<b$, we have
	$\lam_b(a)\wedge a=0\in I$ and so $\lam_b(a)\in I$. Hence, $I$ satisfies condition (\ref{eq3.25}) and so by Lemma \ref{3.25}(v),
	$F_I$ is a filter of $M$. It can be easily seen that $X=\{F\in \mathrm{Fil}(M)\mid F_I\s F,\ 0\notin F\}$ has a maximal element, say $H$,
	(note that by definition $0$ does not belong to $F_I$) which is clearly a maximal filter of $M$.
	Now, by Lemma \ref{3.25}(iii) and the proof of Theorem \ref{3.26}, $I_H$ is a maximal ideal of $M$ containing $I$.

Now let $J_1,J_2$ be two different maximal ideals of $M$ containing $I$. Then there are $x \in J_1\setminus J_2$ and $y \in J_2\setminus J_1$. There is an idempotent $a\in \mathcal I(M)$ such that $x,y \in [0,a]$. Then
$a \notin J_1\cup J_2$ and by Proposition \ref{pr:2}(i), $J_1\cap [0,a]$ and $J_2\cap [0,a]$ are maximal ideals of the $MV$-algebra $([0,a];\oplus,\lambda_a,0,a)$. It is easy to verify that $I\cap [0,a]$ is a proper ideal of $[0,a]$ which is also prime. Using \cite[Cor 1.2.12]{mundici 1}, $J_1 \cap [0,a]=J_2\cap [0,a]$ which implies $x,y \in J_1,J_2$,  an absurd.
\end{proof}

\begin{prop}\label{pr:def}
{\rm (i)} Let $P$ be a prime ideal of an $EMV$-algebra and let $I$ be a proper ideal of $M$ containing $P$. Then $I$ is a prime ideal of $M$.

{\rm (ii)} For each prime ideal $J$ of $M$, the set $\mathcal S(J)=\{I \in \mathrm{Ideal}(M)\mid J \subseteq I\ne M\}$ is a linearly ordered set of prime ideals with respect to the set theoretical inclusion with a top element.
\end{prop}

\begin{proof}
(i) Let $x\wedge y \in I$ for some $x,y \in M$. Since $I$ is a proper ideal of $M$, there is an idempotent $a$ of $M$ such that $x,y \in [0,a]$ and $a\notin I$. Then $I\cap [0,a]$ is a prime ideal of the $MV$-algebra $[0,a]$, $x\wedge y \in I \cap [0,a]$ and $I \cap [0,a]$ is an ideal of $[0,a]$. Applying \cite[Thm 1.2.11(i)]{mundici 1}, we have $I \cap [0,a]$ is a prime ideal of $[0,a]$. Then $x\in I \cap [0,a]\subseteq $ or $y\in I \cap[0,a]\subseteq I$ proving $I$ is prime.

(ii) Let $I_1,I_2$ be two proper ideals of $M$ containing $J$ such that $I_1$ and $I_2$ are not comparable. Then there are $x\in I_1\setminus I_2$ and $y \in I_2 \setminus I_1$. We choose an idempotent $a \in \mathcal I(M)$ such that $x,y \in [0,a]$ and $a \notin I_1 \cup I_2$. Then $J\cap [0,a]$ is a prime ideal of the $MV$-algebra $[0,a]$ which is contained in both $I_1\cap [0,a]$ and $I_2 \cap [0,a]$. By \cite[Thm 1.2.11(ii)]{mundici 1}, $I_1\cap [0,a]$ and $I_2 \cap [0,a]$ are comparable ideals of $[0,a]$. If $I_1\cap [0,a]\subseteq I_2 \cap [0,a]$, then $x \in I_2$, an absurd, and dually if $I_2\cap [0,a]\subseteq I_1 \cap [0,a]$, then $y \in I_1$ also an absurd. Then $I_1$ and $I_2$ are comparable. The top element of $\mathcal S(J)$ is a unique maximal ideal of $M$ containing $J$ which is guaranteed by Proposition \ref{3.28}.
\end{proof}

\begin{rmk}\label{3.28.1}
	In the proof of Theorem \ref{3.26} we showed that if $F$ is a maximal filter of a proper $EMV$-algebra $(M;\vee, \wedge,\oplus,0)$, then
$I_F$ is a maximal ideal of $M$. Now, let $I$ be a maximal ideal of $M$. By Lemma \ref{3.25}(v), $F_I$ is a filter of $M$ and so $F_I$ is contained in a maximal filter $H$.
\vspace{1mm}
	\begin{itemize}[nolistsep]
	\item[{\rm(i)}] Since $F_I\s H$, then from definition it follows that $I_{F_I}\s I_H$.
	
	\item[{\rm(ii)}] $I_H\neq M$. Otherwise, if $I_H=M$, then $\mI(M)\s I_H$ and so by Lemma \ref{3.25}(i), $\mI(M)\cap H=\emptyset$. That is,
	$H=\emptyset$.
	
	\item[{\rm(iii)}] Let $x\in I$. Put $a\in M \setminus I$ such that $x<a$. Then $\lam_a(x)\in F_I$ and so $a\in F_I$. By definition,
	$x=\lam_a(\lam_a(x))\in I_{F_I}$. Therefore, $I\s I_{F_I}$.
	\end{itemize}
	From (i), (ii) and (iii) it follows that $I\s I_{I_F}\s I_H \varsubsetneq M$ and so $I=I_H$. Therefore, any maximal
	ideal $I$ of $M$ is of the form $I_H$ for some maximal filter $H$ of $M$.
	
\end{rmk}

\begin{thm}\label{3.29}
Let $(M;\vee, \wedge,\oplus,0)$ be an $EMV$-algebra and let $I$ be a proper ideal of $M$, $a\in M\setminus I$. Then there exists an ideal $P$ of $M$ which is maximal with respect to the property  $I \subseteq P$ and $a\in M\setminus P$. In addition, $P$ is prime.
\end{thm}

\begin{proof}
If $M$ is not proper and $M\ne\{0\}$, $M$ is a non-degenerated  $MV$-algebra, and the statement is well-known \cite[Prop 1.2.13]{mundici 1}.
So let $M$ be proper. By Zorn's lemma, $X=\{J\in\mathrm{Ideal}(M)\mid I \subseteq J,\, a\notin J\}$ has a maximal element, $P$ say. Thus $a\notin P$. We will show that $P$ is prime. Let $x\wedge y\in P$ for some $x,y\in M$. If $x,y\notin P$, then
	$a\in \langle P\cup\{x\}\rangle$ and $a\in\langle P\cup \{y\}\rangle$, so by Corollary \ref{3.16},
	$a\leq z_1\oplus m.x$ and $a\leq z_2\oplus m.y$ for some $m\in\mathbb{N}$ and $z_1,z_2\in P$. It follows that
	\[ a\leq (z_1\oplus m.x)\wedge (z_2\oplus m.y)\leq (z_1\oplus z_2\oplus m.x)\wedge (z_1\oplus z_2\oplus m.y).\]
	Let $b\in\mI(M)$ such that $x,y,z_1,z_2\leq b$. Since $([0,b]; \oplus,\lam_b,0,b)$ is an $MV$-algebra, then
	by \cite[Prop 1.17(i)]{georgescu}, we have
	\[ a\leq (z_1\oplus z_2)\oplus (m.x\wedge m.y)\leq 2m.(z_1\oplus z_2)\oplus m^2.(x\wedge y)\in P
	\]
	which is a contradiction. Therefore, $P$ is prime.
\end{proof}

\begin{cor}\label{co:3.29}
Every proper ideal of an $EMV$-algebra $M$ can be embedded into a maximal ideal of $M$.
\end{cor}

\begin{proof}
Let $I$ be a proper ideal of $M$. By Theorem \ref{3.29}, there is a prime ideal $P$ of $M$ containing $I$. Applying Proposition \ref{3.28}, we have the assertion in question.
\end{proof}

Note that, if $P$ is a prime ideal of an $EMV$-algebra $(M;\vee, \wedge,\oplus,0)$ and $x,y\in M$, then there exists $a\in\mI(M)$ such that $x,y\leq a$.
Since $([0,a]; \oplus,\lam_a,0,a)$ is an $MV$-algebra, then by \cite[Prop 1.1.7]{mundici 1},
\[\lam_a(\lam_a(x)\oplus y)\wedge \lam_a(\lam_a(y)\oplus x)=0,\]
which implies that $\lam_a(\lam_a(x)\oplus y)\in P$ or $\lam_a(\lam_a(y)\oplus x)\in P$.

\begin{thm}\label{3.30}
	Let $(M;\vee, \wedge,\oplus,0)$ be a proper $EMV$-algebra. Then the radical $\mbox{\rm Rad}(M)$ of $M$, the intersection of all maximal ideals of $M$, is the set
\begin{equation}\label{eq:Rad}
\mathrm{Rad}(M)=\{x\in M\mid x\neq 0,\ \exists\, a\in\mI(M): x\leq a\ \& \ \ (n.x\leq \lam_a(x),\  \forall\, n\in\mathbb{N})\}\cup\{0\}.
\end{equation}
\end{thm}

\begin{proof}
	Let $x\in M\setminus\bigcap\{I\mid I\in\mathrm{MaxI}(M)\}$. Then there exists a maximal ideal $I$ such that $x\notin I$.
	By Remark \ref{3.28.1}, there exists a maximal filter $H$ such that $I=I_H$ and so $x\notin I_H$.
	\begin{eqnarray*}
	x\notin I_H &\Rightarrow &\lam_a(x)\notin H (\forall\, a\in \mI(M)\cap \uparrow x)\Rightarrow \lfloor H\cup \{\lam_a(x)\}\rfloor=M\\
	&\Rightarrow & 0\in \lfloor H\cup \{\lam_a(x)\}\rfloor \Rightarrow y\odot \lam_a(x)^n=0 (\exists\, n\in\mathbb{N}\ \exists\, y\in H), \mbox{ by Proposition \ref{3.24} } \\
	&\Rightarrow & \lam_a(x)^n\leq \lam_b(y), \ (\forall\, b\in\mI(M)\ : \ b\geq a,y)\\
	&\Rightarrow & \lam_a(x)^n\in I_H,\mbox{ since $y\in H$ and $b\geq y$, then $\lam_b(y)\in I_H$}.
	\end{eqnarray*}
If for all $n\in\mathbb{N}$, $n.x\leq \lam_a(x)$, then $\lam_a(n.x)\geq x$ and so $\lam_a(x)^n\geq x$ (note that
$([0,a]; \oplus,\lam_a,0,a)$ is an $MV$-algebra). It follows that $x\in I_H=I$ which is a contradiction. Hence
\[\{x\in M\mid x\neq 0,\ \exists\, a\in\mI(M): x\leq a\ \& \ \ (n.x\leq \lam_a(x),\  \forall\, n\in\mathbb{N})\}\cup\{0\}\s \bigcap\{I\mid I\in\mathrm{MaxI}(M)\}.\]

Now, let $x\in M\setminus \mathrm{Rad}(M)$. Then for all $a\in\mI(M)$ with $x\leq a$, there exists $n\in\mathbb{N}$ that $n.x\nleq \lam_a(x)$. It follows that
$\lam_a(\lam_a(n.x)\oplus \lam_a(x))>0$. By  Theorem \ref{3.29}, there is a prime ideal $P$ of $M$ such that
$\lam_a(\lam_a(n.x)\oplus \lam_a(x))\notin P$ and so $\lam_a(x\oplus n.x)\in P$ and $a\notin P$ (otherwise, from
$\lam_a(\lam_a(n.x)\oplus \lam_a(x))\leq a$ we have $\lam_a(\lam_a(n.x)\oplus \lam_a(x))\in P$).
By Proposition \ref{3.28}, there is a maximal ideal $J$ of $M$ containing $P$.
We claim that the maximal filter $J$ induced from Proposition \ref{3.28} does not contain $a$.
Recall that $J=F_H$, where $H$ is a maximal filter of $M$ containing $F_P$. Check
\begin{eqnarray*}
a\in \mI(M),\ a\notin P&\Rightarrow& \lam_a(x)\in F_P (\forall\, x\in P: x<a), \mbox{ by definition} \\
&\Rightarrow & a\in  F_P, \mbox{ since $\lam_a(x)\leq a$  }\\
&\Rightarrow & F_P\s H \Rightarrow a\in H \\
&\Rightarrow & a\notin I_H, \mbox{  by Lemma \ref{3.25}(i)}.
\end{eqnarray*}
So our claim is true. From $\lam_a((n+1).x)\in J$ it follows that $(n+1).x\notin J$ and so $x\notin J$. Hence
$x\notin \bigcap\{I\mid I\in\mathrm{MaxI}(M)\}$. Therefore, $\bigcap\{I\mid I\in\mathrm{MaxI}(M)\}\s \mathrm{Rad}(M)$.
\end{proof}

\begin{rmk}\label{3.31}
	Let $(B;\vee,\wedge)$ be a  generalized Boolean algebra that is not a Boolean algebra and $(M;\oplus,',0,1)$ be an $MV$-algebra. In Example \ref{3.1}(3),
	we showed that $M\times B$ is an $EMV$-algebra. By \cite[Thm. 2.2]{CoDa}, there exists a Boolean algebra $\overline{B}$ such that
	$B$ is a maximal ideal of $\overline{B}$. Clearly, $M\times \overline{B}$ is an $MV$-algebra containing $M\times B$.
	It is straightforward to prove that $M\times B$ is a maximal ideal of $M\times \overline{B}$. Therefore, any $EMV$-algebra of
	the form $M\times B$, where $M$ is an $MV$-algebra and $B$ is a generalized Boolean algebra is a maximal ideal of an $MV$-algebra.
\end{rmk}

\begin{prop}\label{rmkdep}
Let $(M;\vee, \wedge,\oplus,0)$ be an $EMV$-algebra. If there exists $\{a_i\mid i\in \mathbb{N}\}\s\mI(M)$ such that $\{a_1\vee a_2\vee\cdots \vee a_n\mid n\in\mathbb{N}\}$ is a full subset of $\mI(M)$ and $a_i\wedge a_j=0$ for all distinct elements $i,j\in\mathbb{N}$, then the $EMV$-algebra $M$ can be embedded into an $MV$-algebra.
\end{prop}

\begin{proof}
	Let $(M;\vee, \wedge,\oplus,0)$ be an $EMV$-algebra with the mentioned properties. For each $n\in\mathbb{N}$, $v_n:=a_1\vee \cdots \vee _n$ is an idempotent element of $M$ and so $[0,v_n]$ is an $MV$-algebra. Define a map $f:M\ra \prod_{n\in\mathbb{N}}[0,v_n]$ by $f(x)=(x\wedge v_n)_{n\in\mathbb{N}}$ for all $x\in M$. Clearly, $f$ is a one-to-one map which preserves $\vee$, $\wedge$ and $0$. Now, we show that $f$ preserves $\oplus$. Let $n$ be an arbitrary positive integer. We will show that $\pi_n\circ f:M\ra [0,v_n]$ is a homomorphism of $EMV$-algebras, where $\pi_n$ is the $n$-th canonical projection map.	 Put $x,y\in M$. Then there exists $a\in \mI(M)$ such that $x,y,v_n\in [0,a]$. Since $([0,a]; \oplus,\lam_a,0,a)$ is an $MV$-algebra,	then from
	\begin{eqnarray*}
	\pi_n\circ f(x)\oplus \pi_n\circ f(y)&=&(x\wedge v_n)\oplus (y\wedge v_n)=((x\wedge v_n)\oplus y)\wedge ((x\wedge v_n)\oplus v_n)\\
	&=& ((x\wedge v_n)\oplus y)\wedge v_n, \mbox{ since $v_n$ is idempotent } \\
	&=& ((x\oplus y)\wedge (v_n\oplus y))\wedge v_n=(x\oplus y)\wedge v_n\\
	&=&\pi_n\circ f(x\oplus y)
	\end{eqnarray*}
it follows that $f$ preserves $\oplus$. From definition of the unary operation $'$ in the $MV$-algebra $\prod_{n\in\mathbb{N}}[0,v_n]$ and Remark \ref{3.6}, it can be easily seen that $f$ is an $EMV$-algebra homomorphism. Therefore, $M$ can be embedded into an $MV$-algebra.
\end{proof}

\begin{thm}\label{3.33}
	Let $(M;\vee, \wedge,\oplus,0)$ be an $EMV$-algebra. Then $(\mathrm{Ideal}(M);\s)$ is a complete {\em Brouwer} lattice. Consequently, $\mathbb{EMV}$ is a congruence distributive variety.
\end{thm}

\begin{proof}
Indeed, we have:
	
	(i) Clearly, for each family $\{J_i\mid i\in T\}\s \mathrm{Ideal}(M)$, we have
	\[ \bigwedge_{i\in T} J_i=\bigcap_{i\in T}J_i,\quad \bigvee_{i\in T}J_i=\langle \bigcup_{i\in T}J_i\rangle.
	\]
	
	(ii) Let $I$ be an ideal of $M$. Then for each $x\in I\bigwedge (\bigvee_{i\in T}J_i)$ by Proposition \ref{3.15},
	there exist $n\in\mathbb{N}$ and $c_{i_1},\ldots,c_{i_n}\in \bigcup_{i\in T}J_i$ such that $x\leq c_{i_1}\oplus\cdots\oplus c_{i_n}$.
	Put $b\in M$ such that $x\leq b$. Since $([0,b]; \oplus,\lam_b,0,b)$ is an $MV$-algebra, then by \cite[Prop 1.17]{georgescu},
	\[x=x\wedge (c_{i_1}\oplus\cdots\oplus c_{i_n})\leq (x\wedge c_{i_1})\oplus \cdots \oplus (x\wedge c_{i_n})\in \bigvee_{i\in T}(I\cap J_i).
	\]
	Hence, $\bigwedge$ distributes over arbitrary $\bigvee$.
	
	(iii) For each ideal $I$ of $M$, from (i) and (ii) it follows that
	$\max\{J\in \mathrm{Ideal}(M)\mid I\bigcap J=\{0\}\}$ exists and it is denoted by $I^{\bot}$. In addition, $I^\bot =\{x\in M\mid x\wedge y = 0 \mbox{ for all } y \in I\}$.

(iv) Since  $(\mathrm{Ideal}(M);\s)$ is a Brouwer lattice, it is distributive see e.g. \cite[p. 151]{Bly}. Due to Theorem \ref{ER0}, the lattice of congruences on $M$, $(\mathrm{Con}(M); \s)$ is also a Brouwer lattice. Therefore, $\mathbb{EMV}$ is a congruence distributive variety.
\end{proof}

\begin{prop}\label{3.34}
	Let $(M;\vee, \wedge,\oplus,0)$ be an $EMV$-algebra and $I$ be an ideal of $M$. Then
	$I^\bot=\bigvee_{a\in\mI(M)} I_a^{\bot_a}$, where $I_a=I\cap [0,a]$ and $I^{\bot_a}$ is the pseudo complement of
	$I_a$ in the lattice of ideals of the $MV$-algebra $[0,a]$.
\end{prop}

\begin{proof}
Let $a\in\mI(M)$. Then
\begin{equation*}
I_a^{\bot_a}\cap I_a=\{0\}\Rightarrow I_a^{\bot_a}\cap I=\{0\} \Rightarrow I_a^{\bot_a}\s I^{\bot} \Rightarrow
\bigvee_{a\in\mI(M)}I_a^{\bot_a}\s I^\bot.
\end{equation*}	
On the other hand, if $J$ is an ideal of $M$ such that $J\cap I=\{0\}$, then $J_a\cap I_a=\{0\}$ for all $a\in\mI(M)$,
which implies that
$J_a\s I_a^{\bot_a}$. Hence
\[ J=\bigcup_{a\in\mI(M)}J_a\s \bigcup_{a\in\mI(M)}I_a^{\bot_a}\s \bigvee_{a\in\mI(M)}I_a^{\bot_a}.
\]
and so $I^\bot\s \bigvee_{a\in\mI(M)}I_a^{\bot_a}$.
\end{proof}

Now we show that every subdirectly irreducible $EMV$-algebra is linearly ordered similarly as  does every subdirectly irreducible MV-algebra, see \cite[Thm 1.3.3]{mundici 1}.

\begin{prop}\label{pr:1}
Every subdirectly irreducible $EMV$-algebra $M$ is linearly ordered.
\end{prop}

\begin{proof}
If $M=\{0\}$, the statement is clear.
Let $M\ne \{0\}$ be a subdirectly irreducible $EMV$-algebra. Due to Theorem \ref{3.12}, this means that $M$ has the least non-trivial ideal $I$. Let $a>0$ be any idempotent of $M$. Then $[0,a]$ is an ideal of $M$ as well as of the $MV$-algebra $[0,a]$ which yields $I\subseteq [0,a]$ for each $a \in \mI(M)$. It is clear that $I$ is also the least ideal of every $MV$-algebra $[0,a]$. Therefore, every $[0,a]$ is a linearly ordered $MV$-algebra. Let $x,y \in M$. There is an idempotent $a$ of $M$ such that $x,y\le a$. Then $x\le y$ or $y\le x$ as was claimed.
\end{proof}

In Theorem \ref{3.14}, we showed that any simple $EMV$-algebra is an $MV$-algebra and by Theorem \ref{3.18}, we proved that
for each maximal ideal $I$ of an $EMV$-algebra $(M;\vee, \wedge,\oplus,0)$, $M/I$ is an $MV$-algebra. Now, we want to generalize this result.

\begin{thm}\label{3.35}
	Any $EMV$-algebra can be embedded into an $MV$-algebra.
\end{thm}

\begin{proof} We present two proofs.

(1) Let $(M;\vee, \wedge,\oplus,0)$ be an $EMV$-algebra. By Theorem \ref{3.29}, if $\mathcal P(M)$ is the set of all prime ideals of $M$, then $\cap \mathcal P(M)=\{0\}$. So, the natural map $\phi:M\ra \prod_{I\in \mathcal P(M)}M/I$ defined by $\phi(x)=(x/I)_{I\in S}$ is a one-to-one $EMV$-homomorphism. It is  easy to see that $M/I$ is a chain (since $I_a=I\cap [0,a]$ is a prime ideal of the $MV$-algebra $[0,a]$ for all $a\in \mI(M)$ and all prime ideal $I$ of $M$). By Theorem \ref{th:finite}, $M/I$ is an $MV$-algebra. Therefore, $\prod_{I\in \mathcal P(M)}M/I$ is an $MV$-algebra.

(2) By Theorem \ref{3.7}, the class of $EMV$-algebras is a variety. Therefore, due to the Birkhoff Subdirect Representation Theorem, see \cite[Thm 8.6]{burris}, $M$ is a subdirect product of subdirectly irreducible $EMV$-algebras which are in view of Proposition \ref{pr:1} linearly ordered $EMV$-algebras. By Theorem \ref{th:finite}, every linearly ordered $EMV$-algebra is an $MV$-algebra which gives the result.
\end{proof}


From \cite[Thm. 2.2]{CoDa},  Theorem \ref{th:clan1} and Corollary \ref{co:clan1} we conclude that every proper generalized Boolean algebra, every proper $EMV$-clan and every proper semisimple $EMV$-algebra can be embedded into an $MV$-algebra as a maximal ideal of the $MV$-algebra. In the following, we present a basic result saying that this is true for each proper $EMV$-algebra. 

\begin{thm}\label{th:embed}{\rm [Basic Representation Theorem]}
Every $EMV$-algebra $M$ is either an $MV$-algebra or $M$ can be embedded into an $MV$-algebra $N$ as a maximal ideal of $N$.
\end{thm}

\begin{proof}
If $M$ possesses a top element, then $M$ is an $MV$-algebra, see Example \ref{3.1}(4). If $M$ has no top element, $M$ is proper and according to Theorem \ref{3.35}, there is an $MV$-algebra $N$ such that $M$ can be embedded into $N$ as an $EMV$-subalgebra of the $EMV$-algebra $N$. Let $1$ be the top element of $N$ and without loss of generality, we can assume that $M$ is an $EMV$-subalgebra of $N$. Let $N_0(M)$ be the least $EMV$-subalgebra of $N$ containing both $M$ and the element $1$. In what follows, we will use ideas of the proof of Theorem \ref{th:clan1} to describe $N_0(M)$.

For each $x \in N$, let $x^*= \lambda_1(x)$.

Set
\begin{equation}\label{eq:embed}
N_0:=\{x\in N_0(M) \mid \exists\, x_0\in M\ \text{\rm such that either } x=x_0 \ \text{\rm  or } x = x_0^*\}.
\end{equation}
We assert $N_0(M)=N_0$.

Clearly $N_0$ contains $M$ and $1$. Let $x,y \in N_0$. We have three cases: (i) $x=x_0,y=y_0 \in M$. Then $x\vee y,x\wedge y, x\oplus y \in N_0$. (ii) $x=x^*_0$, $y=y_0^*$ for some $x_0,y_0\in M$. Then $x\vee y = x^*_0\vee y^*_0= (x_0\wedge y_0)^*$, $x\wedge y= (x_0\vee y_0)^*$ and $x\oplus y = x^*\oplus y^* = (x_0\odot y_0)^*\in N_0$. (iii) $x=x_0$ and $y=y_0^*$ for some $x_0,y_0\in M$. Then
$$
x\oplus y = x_0\oplus y^*_0 = (y_0\odot x_0^*)^*= (y_0\odot (x_0\wedge y_0)^*)^*=(y_0 \odot \lambda_b(x_0\wedge y_0))^*,
$$
where $b$ is an idempotent of $M$ such that $x_0,y_0 \le b$; for the last equality we use equality (\ref{eq:x<y}) of Lemma \ref{le:x<y}. Using again Lemma \ref{le:x<y}, we have $y_0 \odot \lambda_b(x_0\wedge y_0) \in M$ so that $x\oplus y \in N_0$.

In addition, if we apply equality (\ref{eq:x<y}), we have
$$x\wedge y = x\wedge y^*_0= x_0^*\odot (x^*_0\oplus y^*_0) = x_0\odot (x_0\odot y_0)^*= x_0\odot \lambda_a(x_0\odot y_0),
$$
where $a$ is an idempotent of $M$ such that $x,y \le a$. So that $x\wedge y \in N_0$. Using $x\vee y = (x^*\wedge y^*)^*$, we have, $x\vee y \in N_0$.

We have just proved that $N_0$ is an $EMV$-algebra containing $M$ and $1$, so that $N_0$ is an $MV$-algebra contained in $N_0(M)$. Therefore, $N_0=N_0(M)$ and $N_0$ contains $M$ properly.

Now we prove that $M$ is a maximal ideal of $N_0$. Since $M$ is a proper $EMV$-algebra, $M$ is a proper subset of $N_0$.  To show that $M$ is an ideal it is sufficient assume $y\le x \in M$. If $y=y_0^*$, this is impossible while $1\notin M$. Therefore, $M$ is a proper ideal of $N_0=N_0(M)$. Now let $y \in N_0\setminus M$, then $y=y_0^*$ for some $y_0\in M$. Then $y^*=y_0 \in M$ showing $M$ is a maximal ideal of the $MV$-algebra $N_0$.
\end{proof}

It is important to note that the converse to Theorem \ref{th:embed}, i.e. whether a maximal ideal of an $MV$-algebra is an $EMV$-algebra, is not true, in general. Indeed, if we take the Chang $MV$-algebra $N=\Gamma(\mathbb Z \lex \mathbb Z, (0,1))$, where $\mathbb Z \lex \mathbb Z$ denotes the lexicographic product of the group of natural numbers $\mathbb Z$ with itself, then the set $I=\{(0,n)\mid n \ge 0\}$ is a unique maximal ideal of $N$, but $I$ is only a $qEMV$-algebra but not an $EMV$-algebra because $I$ has only one idempotent, namely $0=(0,0)$. However, if $M$ is an $MV$-algebra and $I$ is a maximal ideal of $I$ having enough idempotent elements, i.e., for each $x\in I$, there is an idempotent element $a$ of $M$ belonging to $I$ such that $x\le a$, then $I$ is an $EMV$-algebra. It is well known that if $a$ is a Boolean element of $M$, then $[0,a]\subseteq I$ and $([0,a];\oplus_a,^{'a},0,a) $ is an $MV$-algebra, where $x\oplus_a y = (x+y)\wedge a= x\oplus y$, $x^{'a} =a-x=\lambda_a(x)$, $x,y \in [0,a]$. Then due to Theorem \ref{th:embed}, we have that the set $I\cup I'$, where $I'=\{x' \mid x \in I\}$, is the least $MV$-subalgebra of $M$ containing $I$ and $1$.

More about $MV$-algebras for which a proper $EMV$-algebra can be embedded as their maximal ideal will be done at the end of this section, see Theorem \ref{th:equi1}.

Theorem  \ref{3.35} allows us to show that the lattice of all subvarieties of the variety $\mathbb{EMV}$ of $EMV$-algebras is countably infinite similarly as in the case of the lattice of subvarieties of the variety $\mathbb{MV}$ of MV-algebras.

\begin{thm}\label{th:3}
The lattice of subvarieties of the variety $\mathbb{EMV}$ of $EMV$-algebras is countably infinite.
\end{thm}

\begin{proof}
According to Komori \cite{Kom}, the lattice of subvarieties of the variety $\mathbb{MV}$ of MV-algebras is countably infinite. Di Nola and Lettieri presented in \cite{DiLe} an equational base of any subvariety of the variety $\mathbb{MV}$ which consists of finitely many MV-equations using only $\oplus$ and $\odot$.  Hence, let $\mathbb V$ be any subvariety of MV-algebras with a finite equational base $\{f_i(x_1,\ldots,x_n)=g_i(y_1,\ldots,y_m) \mid i=1,\ldots,n\}$, where $f_i,g_i$ are finite MV-terms using only $\oplus$ and $\odot$. Let $\mathcal E(\mathbb V)$ be the subvariety of $EMV$-algebras satisfying equations $f_i(x_1,\ldots,x_n)=g_i(y_1,\ldots,y_m)$ for $i=1,\ldots, n$.

Now let $M$ be any $EMV$-algebra. It generates the subvariety $Var(M)$ of $EMV$-algebras.  According to Theorem \ref{3.35}, there is an $MV$-algebra $N$ such that $M$ can be embedded into $N$.  The MV-algebra $N$ generates the subvariety $\mathbb V(N)$ of MV-algebras, hence, $M$ belongs to the variety $\mathcal E(\mathbb V(N))$ which proves $Var(M)\subseteq \mathcal E(\mathbb V(N))$.

On the other hand, by the proof of Theorem \ref{3.35}, we know that $N$ can be chosen in such a way that $N$ is the direct product of the family $\{M/I\mid I\in \mathcal P(M)\}$. Clearly, $M/I$ as a homomorphic image of $M$ belongs to $Var(M)$ for all $I\in \mathcal P(M)$, and so the direct product $\prod_{I\in \mathcal P(M)}M/I$ also belongs to $Var(M)$, which implies that $N\in Var(M)$. Since any $MV$-algebra is an $EMV$-algebra, then $\mathbb V(N)\s Var(N)\s Var(M)$, where $Var(N)$ is the variety of $EMV$-algebras generated by $N$. Then $\mathcal E(\mathbb V(N))\subseteq Var(M)$ and finally,  $\mathcal E(\mathbb V(N))= Var(M)$.

Now let $\{M_\alpha\mid \alpha \in A\}$ be any system of $EMV$-algebras. For every $M_\alpha$, there is an $MV$-algebra $N_\alpha$ such that $M_\alpha$ can be embedded into $N_\alpha$. Then $Var(M_\alpha)=\mathcal E(\mathbb V(N_\alpha))$ for each $\alpha \in A$ which entails $Var(\{M_\alpha\mid \alpha \in A\})=\mathcal E(\mathbb V(\{N_\alpha\mid \alpha \in A\}))$. Hence, the cardinality of the set of subvarieties of  $\mathbb{EMV}$ is $\aleph_0$.
\end{proof}

\begin{cor}\label{co:3}
For every subvariety $\mathbb V_E$ of the variety $\mathbb{EMV}$, there is a subvariety $\mathbb V$ of $MV$-algebras such that $\mathbb V_E =\mathcal E(\mathbb V)$, and the equational base from {\rm \cite{DiLe}} for $\mathbb V$ is also an equational base for  $\mathbb V_E$.
\end{cor}

\begin{proof}
The statement follows directly from the proof of Theorem \ref{th:3} and \cite{DiLe}.
\end{proof}

For example, (i) the subvariety satisfying the equation $x=0$ is a singleton containing the one-element $EMV$ algebra $\{0\}$. (ii) The equation $x\oplus x = x$ defines the subvariety of generalized Boolean algebras, which is contained in any non-trivial subvariety of $EMV$-algebras. Indeed, if $\mathbb B$ is the variety of Boolean algebras, or equivalently, $\mathbb B$ is the subvariety $MV$-algebras that satisfy equation $x\oplus x = x$, then $\mathcal E(\mathbb B)$ is the subvariety of generalized Boolean algebras. Then $\mathbb B \subseteq \mathbb V$ for any non-trivial variety $\mathbb V$ of $MV$-algebras, and $\mathcal(\mathbb B) \subseteq \mathcal E(\mathbb V)$.
(iii) The equation $x=x$ determines the whole variety $\mathbb{EMV}$.

\section{Categorical Equivalencies} 

In what follows, we present a categorical equivalence of the category of proper $EMV$-algebras with the special category of $MV$-algebras $N$ with a fixed maximal ideal $I$ having enough idempotents and $N=I\cup I'$.

Let $\mathcal{PEMV}$ be the category of proper $EMV$-algebra whose objects are proper $EMV$-algebras and morphisms are homomorphisms of $EMV$-algebras.  Now let $\mathcal{PMV}$ be the category whose objects are couples $(N,I)$, where $N$ is an $MV$-algebra and $I$ is a fixed maximal ideal of $N$ having enough idempotent elements such that $N=I\cup I'$. If $(N_1,I_1)$ and $(N_2,I_2)$ are two objects of $\mathcal{PMV}$, then a morphism in $\mathcal{PMV}$ from $(N_1,I_1)$ into $(N_2,I_2)$ is a homomorphism of $MV$-algebras $\phi: N_1\to N_2$ such that $\phi(I_1)\subseteq I_2$.

If $G$ is an arbitrary Abelian $\ell$-group, then the $MV$-algebra $N=\Gamma(G\lex \mathbb Z,(0,1))$ (perfect $MV$-algebras, see \cite[Sec 7.4]{mundici 1}) has a unique maximal ideal $I=\{(g,0)\mid g \in G^+\}$ and for it we have $I\cup I'=N$. However, $I$ is not an $EMV$ algebra because $(0,0)$ is a unique idempotent of $I$.

We note that if $\phi:(N_1,I_1)\to (N_2,I_2)$ is a morphisms, then $(\phi(N_1),\phi(I_1))$ is an object of $\mathcal{PMV}$, and it is easy to verify that $\mathcal{PEMV}$ and $\mathcal{PMV}$ are indeed categories. In addition, we underline that $\mathcal{PEMV}$ is not a variety, since due to Theorem \ref{3.18}, if $I$ is a maximal ideal of a proper $EMV$-algebra $M$, then $M/I$ is an $MV$-algebra and thus $M/I$ does not
belong to $\mathcal{PEMV}$.

Define a mapping $\Phi: \mathcal{PMV} \to \mathcal{PEMV}$ as follows: For any object $(N,I) \in  \mathcal{PMV}$, let
$$\Phi(N,I):=I$$
and  if $(N_1,I_1)$ and $(N_2,I_2)$ are objects of $\mathcal{PMV}$ and
$\phi: (N_1,I_1)\to (N_2,I_2)$ is a morphism, then
$$\Phi(\phi)(x) := \phi(x),\quad x \in I_1.
$$

\begin{prop}\label{pr:equi1}
$\Phi$ is a well-defined functor that is faithful and full from
the category $\mathcal{PMV}$ into the category $\mathcal{PEMV}$.
\end{prop}

\begin{proof}
First, we show that $\Phi$ is a well-defined functor. In other words, we
have to establish that if $\phi: (N_1,I_1)\to (N_2,I_2)$ is a morphism of proper $EMV$-algebras, then $\Phi(\phi)$ is a morphism in $\mathcal{PEMV}$. Indeed, the mapping $\Phi(\phi)$ is in fact an $EMV$-homomorphism from the $EMV$-algebra $I_1$ into the $EMV$-algebra $I_2$.

Let $\phi_1$ and $\phi_2$ be two morphisms from $(N_1,I_1)$ into $(N_2,I_2)$ such that $\Phi(\phi_1)=\Phi(\phi_2)$. Then $\phi_1(x)=\phi_2(x)$ for each $x \in I_1$. If $x \in N_1\setminus I_1$, then there is an element $x_0 \in I_1$ such that $x= x_0'$. Then $\phi_1(x)=\phi_1(x'_0)= (\phi_1(x_0))'=(\phi_2(x_0))' = \phi_2(x)$ which entails $\phi_1 = \phi_2$, i.e. $\Phi$ is a faithful functor.

To prove that $\Phi$ is a full functor, let $h: I_1 \to I_2$ be a morphism from $\mathcal{PMV}$, i.e. $h$ is a homomorphism of $EMV$-algebras. By Theorem \ref{th:embed}, there are MV-algebras $N_1$ and $N_2$ such that $I_1$ and $I_2$ can be embedded into $N_1$ and $N_2$, respectively, as their maximal ideals. Without loss of generality, we can assume that $I_i$ is a subalgebra of $N_i$ for $i=1,2$. We assert that there is a morphism $\phi: (N_1,I_1)\to (N_2,I_2)$ such that $\Phi(\phi)=h$. In other words $h$ can be extended to an $MV$-homomorphism $\phi$ from $N_1$ into $N_2$ for some objects $(N_1,I_1)$ and $(N_2,I_2)$ from $\mathcal{PMV}$. By (\ref{eq:embed}), $N_1=N_0(I_1)$. So let $x\in N_1\setminus I_1$. There is a unique element $x_0\in I_1$ such that $x=x_0'$. Then we set $\phi(x)=h(x_0)'$. Clearly $\phi(1)=1$, $\phi(x)=h(x)$ if $x \in I_1$, and $\phi(x')=(\phi(x))'$, $x \in N_1$. Now let $x,y \in N_1$. There are three cases: (1) $x,y \in I_1$, then clearly $\phi(x\oplus y)=\phi(x)\oplus \phi(y)$. (2) $x=x_0'$ and $y= y_0'$ for some $x_0,y_0\in I_1$. Then $\phi(x\oplus y)=\phi(x_0'\oplus y_0')=(\phi(x_0\odot y_0))'= (\phi(x_0)\odot \phi(y_0))'=(\phi(x)'\odot \phi(y)')' = \phi(x)\oplus \phi(y)$. (3) $x= x'_0$ and $y=y_0$ for some $x_0,y_0 \in I_1$. There is an idempotent $a \in I_1$ such that $x,y \le a$. Applying (\ref{eq:x<y}) of Lemma \ref{le:x<y}, we get
\begin{eqnarray*}
\phi(x\oplus y)&=&\phi(x_0'\oplus y_0)=(\phi(x_0\odot y_0'))'=
(\phi(x_0\odot (x_0\wedge y_0)'))'= (\phi(x_0\odot \lambda_a(x_0\wedge y_0)))'\\
&=& (\phi(x_0)\odot \lambda_{\phi(a)}(\phi(x_0)\wedge \phi(y_0)))'=
(\phi(x_0)\odot (\phi(x_0)\wedge \phi(y_0))')' = (\phi(x_0) \odot \phi(y_0)')'\\
&=&  \phi(x)\oplus \phi(y).
\end{eqnarray*}
Therefore, $\phi$ is a homomorphism of $MV$-algebras which is an extension of $h$. Whence, $\Phi(\phi)=h$ and $\Phi$ is a full functor.
\end{proof}

\begin{prop}\label{pr:equi2}
Let $M$ be a proper $EMV$-algebra and $h_i:M \to N_i$ be an embedding of $M$ into an $MV$-algebra $N_i$ for $i=1,2$. Then $N_0^i:=N_0(h_i(M_i))$ are isomorphic $MV$-algebras and $(N_i^0,h_i(M))\in \mathcal{PMV}$ for $i=1,2$.
\end{prop}

\begin{proof}
Let $h_i: M \to N_i$ be an embedding for $i=1,2$. By (\ref{eq:embed}) of Theorem \ref{th:embed}, $N_i^0 = N_0(h_i(M))$ for $i=1,2$. Let us define $\psi: N_1^0 \to N_2^0$ such that $\phi(x)=h_2(x_0)$ if $x=h_1(x_0)$ for $x_0\in M_1$ and $\phi(x) = (h_2(x_0))'$ if $x=h_1(x_0)'$ for $x_0 \in M_1$. Then, similarly as in the proof of the Proposition \ref{pr:equi1} that $\Phi$ is a full functor, we can prove that $\phi$ is a homomorphism of $MV$-algebras. In addition, $\phi$ is a bijection, so that it is an isomorphism. Clearly, $(N_i^0,h_i(M))\in \mathcal{PMV}$ for $i=1,2$.
\end{proof}

Let $\mathcal A$ and $\mathcal B$ be two categories and let $f : \mathcal A \to \mathcal B$ be a functor. Suppose that $g,h$ are functors from $\mathcal B$ to $\mathcal A$ such that $g \circ f = id_{\mathcal A}$ and $f \circ h = id_\mathcal B$; then $g$ is a {\it left-adjoint} of $f$ and $h$ is a {\it right-adjoint} of $f$.

\begin{prop}\label{pr:equi3}
The functor $\Phi$ from the category $\mathcal{PMV}$ into the category $\mathcal{PEMV}$ has a left-adjoint.
\end{prop}

\begin{proof}

We claim, for a proper $EMV$-algebra $M$, there is a universal arrow $((N,I), f)$ i.e., $(N,I)$ is an object in $\mathcal{PMV}$ and $f$ is a morphism from $M$ into $\Phi(N,I)=I$ such that if $(N',I')$ is an object from $\mathcal{PEMV}$ and $f'$ is a morphism from $M$ into $\Phi(N',I')$, then there exists a unique morphism $f^* :(N,I)\to (N',I')$ such that
$\Phi(f^*)\circ f=f'$.

Indeed, by Theorem \ref{th:embed} and Proposition \ref{pr:equi2}, there is a unique (up to isomorphism of $MV$-algebras) $MV$-algebra $N$ and an injective $EMV$-homomorphism $f: M \to N$ such that $f(N)$ is a maximal ideal of $N$.  We assert that $((N,I),f)$ is universal arrow for $M$. Let  $(N',I')$ be an object from $\mathcal{PEMV}$ and let $f'$ be a morphism from $M$ into $\Phi(N',I')$. We can define a mapping $f^*: N \to N'$ such that $f^*(f(x)):=f'(x)$ if $x \in M$ and if $y \in N\setminus f(M)$, there is $y_0\in M$ such that $y=(f(y_0))'$, and we set $f^*(y)= (f'(y_0))'$. Then $f^*:N \to N'$ is a unique $MV$-homomorphism such that $\Phi(f^*)\circ f=f'$.

Define a mapping $\Psi: \mathcal{PEMV} \to \mathcal{PMV}$ by $\Psi(M):=(N,I)$ whenever $((N,I),f)$ is a universal arrow for $M$ and if $f':M\to M'$ is an $EMV$-homomorphism, there is a unique morphism $f^*:(N,I)\to (N',I')$, where $\Phi(N',I')=M'$, then we set $\Psi(f'):=f^*$. Using Theorem  \ref{th:embed}, we have that $\Psi$ is a left-adjoint functor of the functor $\Phi$.
\end{proof}

\begin{thm}\label{th:equi1}
The functor $\Phi$ defines a categorical equivalence of the category $\mathcal{PMV}$ and the category of proper $EMV$-algebras $\mathcal{PEMV}$.

In addition, if $h: \Phi(N,I)\to  \Phi(N',I')$ is a morphism of proper $EMV$-algebras, then there is a unique homomorphism $\phi :(N,I)\to (N',I') $ of $MV$-algebras such that we have $h = \Psi(\phi)$, and
\vspace{1mm}
\begin{enumerate}[nolistsep]
\item[{\rm (i)}] if $h$ is surjective, so is $\phi$;
\item[{\rm(ii)}] if $h$ is injective, so is $\phi$.
\end{enumerate}
\end{thm}

\begin{proof}
According to \cite[Thm IV.4.1 (i),(iii)]{MaL}, since $\Psi$ is faithful and full, it is necessary to show that, for any proper $EMV$-algebra $M$ there is an object $(N,I)$ in $\mathcal{PMV}$ such that $\Phi(N,I)$ is isomorphic to $M$. To show that it is sufficient to take any universal
arrow $((N,I),f)$ of $M$.
\end{proof}

Let $(G,u)$ be an Abelian unital $\ell$-group. An $\ell$-{\it ideal} is a convex $\ell$-subgroup $I$ of $G$. An $\ell$-ideal $I$ is {\it maximal} if it is a value of the strong unit $u$, i.e. a maximal proper $\ell$-ideal of $(G,u)$ not containing $u$. Using categorical equivalence between the category of $MV$-algebras and the category of Abelian unital $\ell$-groups, Theorem \ref{functor}, we have by \cite[Thm 7.2.2]{mundici 1} or \cite[Thm 6.1]{Dvu3}: (i) If $I$ is a (maximal) $\ell$-ideal of $(G,u)$, then $I_0=I\cap [0,u]$ is a (maximal) ideal of the $MV$-algebra $N=\Gamma(G,u)$; (ii) If $I_0$ is a (maximal) ideal of $\Gamma(G,u)$, then $I=\{x\in G \mid |x|\wedge u \in I_0\}$ is a (maximal) $\ell$-ideal of $(G,u)$ such that $I\cap [0,u]=I_0$. In addition,
\begin{equation}\label{eq:ideal}
I=\{x \in G \mid  \exists\,  x_i, y_j\in I_0, x = x_1 + \cdots + x_n - y_1-\cdots -y_m\}.
\end{equation}

\begin{prop}\label{pr:equi4}
Let $I_0$ be a maximal ideal of $N=\Gamma(G,u)$ and
let $I$ be a unique maximal $\ell$-ideal of $(G,u)$ generated by $I_0$. We define $I^u=\{nu-y\mid n\ge 1,\, y\in I,\, 0\le y <nu\}$. Then $I_0\cup I_0'=N$ if and only if $G^+=(I^+)\cup I^u$.
\end{prop}

\begin{proof}
It is clear that $I^u=\{x\in G \mid \exists\, y_1,\ldots,y_n \in I_0,\, x=y_1'+\cdots+y_n'\}$. In addition, $I\cap I^u=\emptyset$.

Let $I_0\cup I_0'=\Gamma(G,u)$ and choose $x \in G^+$. Then $x= x_1+\cdots+x_m +(u-y_1)+\cdots + (u-y_n)$, where $x_i,y_j \in I_0$, so that $x=x_0+nu-y_0$, where $x_0=x_1+\cdots + x_m\in I^+$ and $y_0=y_1+\cdots +y_n\in I^+$. If $n=0$, then $x=x_0 \in I^+$. If $n>0$, then $x\wedge u = (x_1+\cdots+x_m +(u-y_1)+\cdots + (u-y_n))\wedge u = x_1\oplus \cdots\oplus x_m \oplus (u-y_1)\oplus \cdots \oplus (u-y_n)\in \Gamma(G,u)$. Then $m=0$, $x_0$ and $x =nu-y_0\in I^u$.

Conversely, let $G^+=I^+ \cup I^n$. Then $N= N\cap G^+= (N\cap I^+)\cup (N\cap I^u)=N \cup (N\cap I^u)$.  Let $x\in N\cap I^u$, then $x=nu-y$ for some $0\le y< nu$, $y \in I$. Using the Riesz Decomposition Property for $\ell$-groups, we have $y =y_1+\cdots+y_n$, where $y_i \in I_0$. Then $x= (nu - y)= ((u-y_1)+\cdots + (u-y_n))\wedge u = y_1'\oplus \cdots \oplus y_n'=(y_1\odot \cdots \odot y_n)'$, but $y_0:=y_1\odot \cdots \odot y_n \in I_0$ so that $x= y_0 \in I_0'$ which yields $N\cap I^u=I_0'$.  Then $N= I_0\cup I_0'$.
\end{proof}

Inspired by the previous categorical equivalence, let $\mathcal{PUALG}$ be the category of unital Abelian $\ell$-groups with a fixed maximal $\ell$-ideal with a special property. Namely, the objects are triples $(G,u,I)$ such that $(G,u)$ is an Abelian unital $\ell$-group and $I$ is a fixed maximal $\ell$-ideal $I$ of $(G,u)$ such that $G^+=I^+\cup I^u$ and the ideal $I_0=I\cap [0,u]$ of $\Gamma(G,u)$ has enough idempotent elements. If $(G_1,u_1,I_1)$ and $(G_2,u_2,I_2)$ are two objects of $\mathcal{PUALG}$, then a mapping $f: (G_1,u_1,I_1) \to (G_2,u_2,I_2)$ is a morphism if $f$ is a homomorphism of unital $\ell$-groups such that $f(I_1)\subseteq I_2$. Our aim is to show that $\mathcal{PUALG}$ is categorically equivalent to the category $\mathcal{PMV}$. We will follows techniques used in the previous categorical equivalence.

Let us define a functor $\Gamma_I: \mathcal{PUALG}\to \mathcal{PMV}$ as follows: if $(G,u,I)$ is an object of $\mathcal{PUALG}$, then
$$
\Gamma_I(G,u,I):=(\Gamma(G,u),I\cap[0,u]),
$$
and if $f$ is a morphism from an object $(G_1,u_1,I_1)$ into another one $(G_2,u_2,I_2)$, then
$$\Gamma_I(f)(x) :=f(x), \quad x \in \Gamma(G,u).
$$

\begin{prop}\label{pr:equi5}
$\Gamma_I$ is a well-defined functor that is faithful and full.
\end{prop}

\begin{proof}
Clearly $\Gamma_I(G,u,I)=(\Gamma(G,u),I\cap [0,u]) \in \mathcal{PMV}$.
If $f: (G_1,u_1,I_1) \to (G_2,u_2,I_2)$ is a morphism, then the restriction of $f$ onto $\Gamma(G_1,u_1)$ is in fact a homomorphism of $MV$-algebras with $f(I_1)\subseteq I_2$, so that $\Gamma_I(f)(I_1\cap [0,u_1])\subseteq I_2\cap [0,u_2]$,  and $\Gamma_I$ is a correctly defined functor.

Let $f_1$ and $f_2$ be two morphisms from $(G_1,u_1,I_1)$ into $(G_2,u_2,I_2)$ such that $\Gamma(f_1)=\Gamma(f_2)$. Then $f_1(x)=f_2(x)$ for each $x\in \Gamma(G_1,u_1)$. Since $f_i$ for $i=1,2$ is a homomorphism of unital $\ell$-groups, it is easy to see that $f_1(x)=f_2(x)$ for each $x \in G_1$ and $f_1=f_2$.

Now let $\mu: \Gamma_I(G_1,u_1,I_1)=(\Gamma(G_1,u_1),I_1\cap [0,u_1])\to \Gamma_I(G_2,u_2,I_2)=(\Gamma(G_2,u_2),I_2\cap [0,u_2])$ be a morphism, i.e. $\mu$ is an $MV$-homomorphism from $\Gamma(G_1,u_1)$ into $\Gamma(G_2,u_2)$ such that $\mu(I_1\cap [0,u_1])\subseteq I_2\cap [0,u_2]$. Using methods of the proof of \cite[Prop 6.1]{Dvu2}, we can uniquely extend $\mu$ to a homomorphism of unital $\ell$-groups $f:G_1 \to G_2$. Since $I_i\cap [0,u_i]$ can be uniquely extended to the $\ell$-ideal $I_i$, $i=1,2$, we have that $f$ is a morphism from $(G_1,u_1,I_1)$ into $(G_2,u_2,I_2)$, which proves $\Gamma_I$ is a full functor because $\Gamma_I(f)=\mu$.
\end{proof}

Now we introduce the following notions. On every $MV$-algebra $N$ we can define a partial addition $+$ such that $x+y$ is defined iff $x\odot y = 0$, and in such a case, $x+y:=x\oplus y$; if $N =\Gamma(G,u)$, then the partial addition coincides with the group addition related to $[0,u]$. We say that a couple $(G,f)$ is a {\it universal group} for an $MV$-algebra $N$ if (i) $f$ is a mapping from $M$ into a po-group $G$ which preserves partial addition $+$ on $N$ such that $G= G^+-G^+$, $f(M)$ generates $G^+$ as a semigroup, (ii) for any group $K$ and any $+$-preserving mapping $h:N\to K$, there is a group homomorphism $\phi: G \to K$ such that $h = \phi \circ f$. Due to \cite[Thm 5.3]{Dvu2} if $N\cong \Gamma(G,u)$, then $(G,f)$ is a universal group for $N$, where $f$ is an isomorphism $f: N \to \Gamma(G,u)$.

\begin{prop}\label{pr:equi6}
The functor $\Gamma_I$ from the category $\mathcal{PUALG}$ into the category $\mathcal{PMV}$ has a left-adjoint.
\end{prop}

\begin{proof}
We assert that for an object $(N,I_0)$, there is a universal arrow $((G,u,I),f)$, i.e. $(G,u,I)$ is an object from $\mathcal{PUALG}$ and $f$ is a morphism from $(N,I_0)$ into $\Gamma_I(G,u,I)=(\Gamma(G,u),I\cap [0,u])= (\Gamma(G,u),I_0)$ such that if $(G',u',I')$ is an object from $\mathcal{PUALG}$ and $f'$ is a morphism from $(N,I_0)$ into $\Gamma_I(G',u',I')=(\Gamma(G',u'),I'\cap [0,u'])$, then there is a unique morphism $f^*: (G,u,I) \to (G',u,I')$ such that $\Gamma_I(f^*)\circ f = f'$.

Take the universal group $(G,f)$ for the $MV$-algebra $N$. Then $f$ is an $MV$-bijection from $M$ onto $\Gamma(G,u)$. We assert $((G,u,I),f)$ is a universal arrow for $(N,I_0)$, where $I$ is an $\ell$-ideal of $G$ generated by $f(I_0)$. Indeed, take an object $(G',u',I')$ from $\mathcal{PUALG}$ and let $f'$ be a morphism from $(N,I_0)\to \Gamma_I(G',u',I')= (\Gamma(G',u'),I'_0)$, where $I'_0=I'\cap [0,u']$. Since $f:M \to \Gamma(G,u) \subset G^+$ is a $+$-preserving mapping and $f':M \to \Gamma(G',u')\subseteq G'$ is also a $+$-preserving mapping, then there is a unique homomorphism of unital $\ell$-groups $f^*: G\to G'$ such that $f^*\circ f=f'$. First $x \in I_0$, then $f(x)\in I$ and thus $f^*(f(x))=f'(x)\in I'_0$. If we take a general $x\in I$, see (\ref{pr:equi4}), then $f'(x) \in I'$, so that $f^*$ is also a morphism from $(G,u,I)$ to $(G',u,'I')$, i.e. $((G,u,I),f)$ is a universal arrow in question.

Define a mapping $\Xi_I: \mathcal{PMV}\to \mathcal{PUALG}$ by $\Xi_I(N,I_0)=(G,u,I)$ if $((G,u,I),f)$ is a universal arrow for $(N,I_0)$ and $I$ is a maximal $\ell$-ideal of $G$ generated by $f(I_0)$. If $f'$ is a morphism from $(N,I_0)$ into $(N',I'_0)$, there is a unique morphism $f^*: (G,u,I) \to (G',u',I')$, where $N'\cong\Gamma(G',u')$ and $I'$ is a maximal $\ell$-ideal of $G'$ generated by $f'(I_0')$, then $\Xi_I(f'):= f^*$. Therefore, $\Xi_I$ is a left-adjoint of $\Gamma_I$.
\end{proof}

\begin{thm}\label{th:equi7}
The functor $\Gamma_I$ defines a categorical equivalence of the category $\mathcal{PUALG}$ and the category $\mathcal{PMV}$.
\end{thm}

\begin{proof}
The statement follows from \cite[Thm IV.4.1(i),(iii)]{MaL} and Propositions \ref{pr:equi5}--\ref{pr:equi6}.
\end{proof}

\begin{cor}\label{co:equi8}
The categories $\mathcal{PUALG}$, $\mathcal{PMV}$ and $\mathcal{PEMV}$ are mutually categorically equivalent.
\end{cor}

\begin{proof}
It follows from Theorem \ref{th:equi1} and Theorem \ref{th:equi7}.
\end{proof}

\section{Conclusion}

We have introduced the notion of an $EMV$-algebra, Definition \ref{de:GMV}, which generalizes the notion of an $MV$-algebra and of a generalized Boolean algebra. We have exhibited its basic properties and notions as ideals, congruences, filters, and their mutual relationship, Theorem \ref{3.12}. Nevertheless an $EMV$-algebra $M$ has not necessarily a top element, $M$ has a maximal ideal, Theorem \ref{3.26}. We have defined an $EMV$-clan as an $EMV$-algebra of fuzzy sets. We have shown that every $EMV$-algebra is semisimple iff it is isomorphic to some $EMV$-clan of fuzzy sets, Theorem \ref{th:semis}. A state-morphism is any $EMV$-homomorphism from $M$ into the $MV$-algebra of the real interval $[0,1]$ which attains the value 1. State-morphisms are in a one-to-one relationship with maximal ideals of $M$, Theorem \ref{th:state}.

We have shown that every $EMV$-algebra can be embedded into an $MV$-algebra, Theorem \ref{3.35}. Theorem \ref{th:embed} characterizes any $EMV$-algebra saying that either it is an $MV$-algebra or it can be embedded into an $MV$-algebra as its maximal ideal.

The class of $EMV$-algebras forms a variety, Theorem \ref{3.7}. Using the equational base of any subvariety of the variety of $MV$-algebras, \cite{DiLe}, we describe a functional base of any subvariety of the variety $\mathbb{EMV}$ of $EMV$-algebras, Corollary \ref{co:3}, and the cardinality of all subvarieties of the variety $\mathbb{EMV}$ is $\aleph_0$, Theorem \ref{th:3}. Finally, we presented mutually categorical equivalencies of the category of proper $EMV$-algebras, a special category of $MV$-algebras $N$ with a fixed maximal ideal $I$ having enough idempotents, and a special categories of Abelian unital $\ell$-groups, Theorem \ref{th:equi1}, Theorem \ref{th:equi7} and Corollary \ref{co:equi8}.

With the present paper we have opened a new and interesting window into the realm of unbounded generalizations of $MV$-algebras and generalized Boolean algebra, and we hope to continue in this research, for example with a variant of the Loomis--Sikorski theorem for $\sigma$-complete $EMV$-algebras.

\end{document}